\documentclass[10pt,reqno]{amsart}
\usepackage{amsmath}
\usepackage{amsfonts}
\usepackage{amsthm}
\usepackage{amssymb}
\usepackage{enumerate}
\usepackage{color}
\usepackage{graphicx}
\usepackage{float}
\usepackage[colorlinks,linkcolor=blue]{hyperref}
\allowdisplaybreaks
\theoremstyle{plain}
\newcommand{\blue}{\color{blue}}
\newcommand{\red}{\color{red}}
\newtheorem{thm}{Theorem}[section]
\newtheorem{prop}[thm]{Proposition}

\newtheorem{cor}[thm]{Corollary}

\newtheorem{lemma}[thm]{Lemma}
\theoremstyle{definition}
\newtheorem{defn}[thm]{Definition}
\theoremstyle{remark}
\newtheorem{rmk}[thm]{Remark}

\newtheorem*{thmn}{Theorem}

\makeatletter

\newcommand{\Rmnum}[1]{\expandafter\@slowromancap\romannumeral #1@}
\makeatother

\begin{document}
\title[]
{\bf Topology of Surfaces with Finite Willmore Energy}
\author[ ]
{Jie Zhou}
\address{
\newline
Department of Mathematical Sciences, Tsinghua University, Beijing, P. R. China, 100084
{\tt Email: zhoujiemath@mail.tsinghua.edu.cn}}
\today
\maketitle
\begin{abstract}
  In this paper, we study the critical case of the Allard regularity theorem. Combining with Reifenberg's topological disk theorem, we get a critical Allard-Reifenberg type regularity theorem. As a main result, we get the topological finiteness for a class of  properly immersed surfaces in $\mathbb{R}^n$ with finite Willmore energy. Especially, we prove a removability of singularity of multiplicity one surface with finite Willmore energy and a uniqueness theorem of the catenoid under no a priori  topological finiteness assumption.
\end{abstract}
\tableofcontents
\section{Introduction}
Assume $\Sigma\subset \mathbb{R}^n$ is a properly immersed smooth surface and denote the immersion by $f:\Sigma\to \mathbb{R}^n$. Let $g=f^*g_{\mathbb{R}^n}$ be the induced metric and $H_f=\triangle_g f$ be the mean curvature. If $H_f=0$, $f$ is called a minimal immersion and  $\Sigma$ is called an immersed minimal surface in $\mathbb{R}^n$.  One of the most important property for minimal surfaces in $\mathbb{R}^n$ is the monotonicity formula, i.e., for  $x\in \mathbb{R}^n$,
$$\Theta(x,r)=\frac{\mathcal{H}^2(B_r(x)\cap \Sigma)}{\pi r^2}$$
is increasing, where $\mathcal{H}^2$ is the two dimensional Hausdorff measure in $\mathbb{R}^n$.  It implies the density
$$\Theta(\Sigma,\infty)=\lim_{r\to +\infty}\Theta(x,r)\in [1,\infty]$$
of a minimal surface at infinity is well defined. A first important fact about the density of minimal surface is the following corollary of the Allard regularity theorem\cite{A72}: if an immersed minimal surface satisfying $\Theta(\Sigma,\infty)<1+\varepsilon$ for $\varepsilon$ sufficient small, then $\Sigma$ is a plane.  For $\Theta(\Sigma,\infty)=2$, in the case $n=3$, there are two typical nontrivial examples--- the catenoid($x_1^2+x_2^2=ch^2x_3$) and Scherk's singly-periodic surface.
\begin{figure}[!htbp]
	\centering
	\begin{tabular}{cc}
		\includegraphics[width=0.40\linewidth]{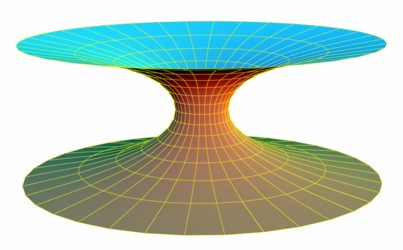}  &
        \hspace{0.3in}
		\includegraphics[width=0.20\linewidth]{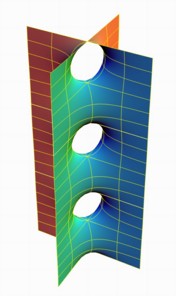}   \\
		(Catenoid) & (Scherk's singly periodic surface)\\
	\end{tabular}

All pictures of minimal surfaces in this paper are taken from \url{www.indiana.edu/~minimal}.
\end{figure}
 They are both embedded minimal surfaces. The catenoid is rotationally symmetric, is the simplest minimal surface except for the plane and can be regarded as the fundamental solution of minimal surface equation. The catenoid has finite topology and finite total curvature but Scherk's singly periodic surface has infinite topology and infinite total curvature. And it is found by Karcher\cite{K88} that there is a one parameter deformation $\Sigma_\theta, \theta\in (0,\frac{\pi}{2}]$, of Scherk's surface $\Sigma_{\frac{\pi}{2}}$. They are all embedded minimal surfaces with $\Theta(\Sigma_{\theta},\infty)=2$ and are also called Scherk's surfaces. Conversely, Meeks and Wolf proved:
\begin{thmn}[\bf{Meeks-Wolf,\cite{MW07}}] A connected properly immersed minimal surface in $\mathbb{R}^3$ with infinite symmetry group and $\Theta(\Sigma,\infty)<3$ is a plane, a catenoid or a Scherk singly-periodic minimal surface $\Sigma_\theta,\theta\in (0,\frac{\pi}{2}]$.
\end{thmn}
They conjecture the infinite symmetry condition can be removed(see also Conjecture 10 in\cite{M03}). For $3\le\Theta(\Sigma,\infty)<\infty$, there are not so clear classification, and Meeks and Wolf also conjecture such minimal surfaces admit unique tangent cone at infinity\cite[Conjecture 1]{MW07}.

Besides the above uniqueness result of Meeks and Wolf, there are many classical classification theorems for minimal surfaces\cite{O64}\cite{O69}\cite{LR91}\cite{L92}\cite{C84}\cite{C89}\cite{C91}\cite{S83}\cite{HM90}. Their common requirement is the minimal surface has finite total curvature, i.e., $$\int_{\Sigma}|A|^2d\mathcal{H}^2<\infty,$$ where $A$ is the second fundamental form of the surface. Especially, by moving plane method, Schoen\cite{S83} proved the only connected complete immersed minimal surface in $\mathbb{R}^3$ with finite total curvature and two embedding ends is the catenoid.  There is a purely topological description for embedded minimal surface with finite total curvature. A surface is said to have finite topology if it is homeomorphic to a closed surface with finite many points removed. And the number of ends of a properly immersed minimal surface is defined by the number of the noncompact connected components of the surface at infinity, i.e.,
   $$e(\Sigma,\infty)=\lim_{r\to \infty}\tilde{\beta}_0(\Sigma\cap(\mathbb{R}^n\backslash B_r(0)))\in [1,\infty],$$
   where by $\tilde{\beta}_0$ we mean the number of noncompact connected components of a topology space. Each such noncompact connected component at infinity is called an end of $\Sigma$. On the one hand, by Huber's result\cite{H57}, any surface in $\mathbb{R}^n$ with finite total curvature must has finite topology. On the other hand, with Meeks and Rosenberg's\cite{MR93} classification of the complex structure of properly embedded minimal surface with at least two ends, Collin \cite{C97} proved a properly embedded minimal surface in $\mathbb{R}^3$ with at least two ends has finite total curvature if and only if it has finite topology. In both \cite{MR93} and \cite{C97}, the assumption $e(\Sigma,\infty)\ge 2$ is necessary to rule out the helicoid type ends. To distinguish the number of ends is also helpful for understanding Meek's conjecture. The catenoid has two ends. But the ``two" tangent planes of Scherk's singly-periodic surface joint together, which forces the surface to possess only one end. So a corollary of Meek's conjecture is that the only connected properly immersed minimal surface in $\mathbb{R}^3$ with $\Theta(\Sigma,\infty)<3$ and at least two ends is the catenoid. By the results of Schoen and Collin recalled above, the only gap to the corollary is the topological finiteness of the surface. And the topological finiteness is the main question we care about in this paper:

   \bigskip

   {\blue For a surface immersed in $\mathbb{R}^n$ with finite Willmore energy $\int_{\Sigma}|H|^2d\mathcal{H}^2<\infty$ (or simply, $H=0$), when does it have finite topology?}

   \bigskip

   The counterexample of Scherk's singly periodic minimal surface gives some geometric intuition: The number of ends should not be too less with respect to the density. Otherwise, ``different" tangent planes at infinity will twist together to shape infinite many genuses.  And our answer is:
   \begin{thm}[$\mathbf{Finite\ Topology}$]\label{main}Assume $\Sigma\subset \mathbb{R}^{2+k}$ is a properly immersed  open surface with finite Willmore energy, i.e.,
    $$\int_{\Sigma}|H|^2d\mathcal{H}^2<\infty.$$
     If its number of ends is not less than the lower density at infinity, more precisely,
   $$e(\Sigma,\infty)>\Theta_*(\Sigma,\infty)
   -1<+\infty,$$
   then $\Sigma$ has finite topology and finite total curvature and $\Theta(\Sigma,\infty)=e(\Sigma,\infty)$ is an integer number.
   \end{thm}

    By some geometric measure theory argument\cite{KLS}(see Remark \ref{end density}), the assumption $e(\Sigma,\infty)>\Theta_*(\Sigma,\infty)-1$ in fact implies $\Sigma$ has exact $e=e(\Sigma,\infty)=\Theta(\Sigma,\infty)$ many ends and each of them has density one at infinity. By the compactness theorem for integral varifolds\cite{A72}, these ends blow down to planes with multiplicity one. Thus by Leon Simon's theorem on the uniqueness of tangent cone with smooth cross section\cite{LS83b}\cite[page 269, The paragraph after Theorem 5.7]{LS85}, in the case of
    $$H=0,$$
      each end of $\Sigma$ is a graph over a tangent plane, hence already has finite topology.

 In\cite{LS83b} and \cite{LS85}, by using the variation structure and PDE techniques, especially the monotonicity formula and the $3$-circle theorem,  Leon Simon established a decay estimate around the isolated singularities of solutions for very general variation equations and got the uniqueness of the tangent cone at isolated singular points. Leon Simon also showed\cite{LS85} the same method works for the tangent cone at infinity. This method is very powerful in analysing the asymptotic behavior of Geometric PDE. For decades, the general method has been applied to many geometric objects including minimal surfaces, harmonic maps, Einstein metrics and corresponding geometric flows. These conclusions imply much more analytic information than the topological finiteness. And we are trying to understand if only caring about the topology, can we get a soft result under looser condition without equation. Theorem \ref{main} is the answer. Below we still take the case of $H=0$ to explain our key observation. It will not loss generality.

   The idea comes out when we are watching minimal surfaces by  the inversion.  By combining the monotonicity formulae of a minimal surface $\Sigma$ and its inverted surface $\tilde{\Sigma}$ and a key conformal antisymmetrical invariance we observe(see (\ref{localinvariant})), we get the following density identity:
    \begin{align*}
   \Theta(\tilde{\Sigma},p)
   =\frac{1}{16\pi}\int_{\tilde{\Sigma}\backslash\{p\}}|\tilde{H}|^2d\mu_{\tilde{g}}
   =\Theta(\Sigma,\infty), \ \ p\notin \Sigma,
  \end{align*}
  which means the single quantity $\Theta(\Sigma,\infty)$ can control both the Willmore energy $\int_{\tilde{\Sigma}\backslash\{p\}}|\tilde{H}|^2d\mu_{\tilde{g}}$ and the local density $\Theta(\tilde{\Sigma},p)$ of $\tilde{\Sigma}$ at the inverting base point $p$.
  \begin{figure}[H]
	\centering
	\begin{tabular}{cc}
		\includegraphics[width=0.25\linewidth]{catenoid2.jpg}  &
        \hspace{0.05in}
		
		\includegraphics[width=0.25\linewidth]{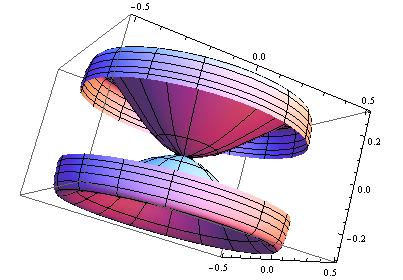}\\
		(Catenoid) &(Part of inverted catenoid)\\
	\end{tabular}
   \end{figure}
   This implies if we invert only one end with density one, then we will get a varifold with density one at the inverting point and bounded Willmore energy, which is on the border of the classical Allard regularity theorem. Recall the Allard regularity theorem\cite{A72} says if an integral $n-$varifold $V=\underline{v}(M,\theta)$ in $B_r(0)\subset\mathbb{R}^{n+k}$ satisfies
  \begin{align*}
  \Theta^n(V,0)<1+\varepsilon, \ \ \ \ \ \ \ \ \ (r^{p-n}\int_{B_r(0)}|H|^p)^{1/p}<\varepsilon
  \end{align*}
  for some $p>n$, $\varepsilon$ small and $0\in sptV$, then the varifold is a $C^{1,\alpha=1-\frac{n}{p}}$ graph in a small neighborhood of $0$. For a smooth immersion $f:M^n\to \mathbb{R}^{n+k}$, $H=\triangle_g f$.  Comparing a varifold to a function, then the generalized mean curvature should be regarded as the weak ``Laplacian".  In this viewpoint, Allard regularity theorem could be regarded as a geometric  nonlinear disturbed version of the $W^{2,p}$ estimates for solutions of linear elliptic equations, combining with the Sobolev embedding theorem
  $W^{2,p}\hookrightarrow C^{1,1-\frac{n}{p}}$.
  \begin{table}[H]
\begin{tabular}{|c|c|c|c|c|}
\hline
Geometry & smooth manifold&Varifold& weak $H$ &  {\blue Allard Regularity}\\
\hline
Analysis&smooth function& Sobolev function & $\triangle_{dist} f$ & $W^{2,p}$ Esitimate\\
\hline
\end{tabular}
\end{table}

   But the mean curvature equation is nonlinear, when getting regularity, one need to do linear approximation first and then use a supercritical index (here $p>n$) to get an iteration program and then a Campanato type regularity estimate. And now it is in the critical case $p=n=2$. The best expected result is a regularity of type ($W^{2,2}\hookrightarrow$) $C^{\alpha}$. By the experience of graphical estimate(See \cite[Lemma 2.11]{CM04} or \cite[Lemma 2.4]{CM11}), the graphical result is always corresponding to a Lipschitz estimate, which seems impossible in our case. So we may not get a $C^{\alpha}$ graph but only get a $C^{\alpha}$ parametrization, which is also enough to show the end is embedded and has finite topology. There is another positive evidence. In \cite{SZ}, Sun and the author proved a properly immersed smooth surface in the unit ball with finite area and small total curvature admits $C^{\alpha}$ parametrization with uniform estimate in some uniform small scale, which can be regarded as a geometric disturbed version of Sobolev's embedding $W^{2,2}\hookrightarrow C^{\alpha}$. This indicates the $C^{\alpha}$ parametrization is hopeful and encourages us to check the original proof of Allard regularity theorem in the critical case. It turns out there is no difficulty in getting the Lipschitz approximation\cite[Section 20]{LS83} from Leon Simon's monotonicity formula\cite{LS93}, but it is impossible to run the iteration program to get a decay of the tilt-excess\cite[Section 22]{LS83}. 

    Fortunately, there is a well developed criteria for the $C^{\alpha}$ regularity of a closed set in $\mathbb{R}^n$. That is Reifenberg's topological disk theorem\cite{R60}\cite{M66}\cite{LS96}, whose proof contains a geometric iteration program.
    Reifenberg's theorem has been established in 1960. Recent twenty years, many mathematicians used the method to research the regularity of both Ricci limit spaces and Radon measures. Let us refer \cite[Appendix]{CC97} \cite{DKT01}\cite{P98}\cite[section 7]{DP02}\cite{DP07} for readers who are interested in related topics. Especially,  Paolini proved\cite{P98} the $C^{\alpha}$ regularity for minimal boundaries in $\mathbb{R}^n$ with mean curvature in $L^n$. Similarly, in our critical case, when combining with the Lipschitz approximation, we can check Reifenberg's condition. As a result, we get the $C^{\alpha}$--regularity for rectifiable $2$-varifold with square integrable generalized mean curvature at those points with density close to one. See Theorem \ref{Holder Regularity} for precise statement.


As an application of Theorem \ref{main}, we studied isolated singularities for properly immersed surfaces with finite Willmore energy. We get the removability of such singularities under the assumption of density less than two.  We do not assume the surface have finite topology or finite total curvature a priori. See Corollary \ref{removability} for details.  As corollaries of Theorem \ref{main}, we also give a simple proof of the uniqueness of the catenoid(see Corollary \ref{catenoid}) and analysis the structure of minimal ends in $\mathbb{R}^{2+k}$ with multiplicity less than two.


 This paper is organized as following. In section \ref{lipschitz approximation}, we prove the Lipschitz approximation theorem. In section \ref{reifenberg}, we check the Reifenberg condition and complete the proof of the $C^\alpha$ regularity. In section \ref{application}, we deduce the density identity of inverting minimal surface and apply the $C^{\alpha}$ regularity theorem to ends with density less than two to get the main theorem of this paper. In section \ref{realapplication}, we give the two applications.  
\section{Lipschitz Approximation}\label{lipschitz approximation}
In this section, we check out the Lipschitz approximation theorem in the critical case. It is the first step of proving the Allard regularity theorem and many of the ideas are similar to those of \cite{LS83}(see also \cite[section 5.2, 5.4, 5.5, 5.6]{LS14}), except for a careful analysis involving the remainder term of Leon Simon's monotonicity identity (\ref{monotonicity equality}) in the proof Lemma \ref{band} and some other details. We also focus on the semi-Reifenberg condition (\ref{semireifenberg}), which is essential for the proof of the $C^\alpha$ regularity theorem in section \ref{reifenberg}.

For a rectifiable $2$-varifold $V=\underline{v}(\Sigma,\theta)$ in an open set $U\subset \mathbb{R}^n$, we always denote the corresponding Radon measure by $\mu=\mu_V=\mathcal{H}^2\llcorner \theta$, i.e, for any Borel set $A\subset \mathbb{R}^n$,
$$\mu(A)=\mu_V(A)=\int_{A\cap \Sigma}\theta d\mathcal{H}^2.$$

The following is the main result of this section---the Lipschitz approximation theorem.
\begin{thm}[\textbf{Lipschitz Approximation for 2-varifold}]\label{Lipschitz Approximation for 2-varifold}
Assume $V=\underline{v}(\Sigma,\theta)$ is a rectifiable 2-varifold in $U\supset B_{\rho}(0)\subset \mathbb{R}^{2+k}$ with generalized mean curvature $H\in L^2(d\mu_V)$,  $0\in sptV$ and $\theta\ge 1$ for $\mu_V-a.e. x\in U$. Then there exists small $\delta_6(=\frac{1}{2^{1688}k^{40}})$ such that for any $\delta\le \delta_6$ if
$$\frac{\mu_V(B_{\rho}(0))}{\pi \rho^2}\le1+\delta\text{\  and \ }\int_{B_{\rho}(0)}|H|^2\le \delta,$$
then for any $\xi\in B_{\frac{1}{2}\delta^{\frac{1}{2}}\rho}(0)$ and  $\sigma\in (0,\frac{1}{2^{16}}\delta^{\frac{1}{2}}\rho)$, there exist $T=T(\xi, \sigma)\in G_{2+k,2}(\mathbb{R})$ passing through $\xi$ and a Lipschitz function $$f=(f^1,f^2,...f^k):B_{\sigma}(\xi)\cap T\to \mathbb{R}^k:=T^{\bot}$$
 with
\begin{align*}
   &i) \ \ Lipf\le\delta^{\frac{1}{40}},\\
   &ii) \ \ \operatorname*{sup}\limits_{x\in B_{\sigma}(\xi)}|f(x)|\le\delta^{\frac{1}{40}}\sigma,\\
   &iii)\ \ \operatorname*{sup}\limits_{x\in B_{\sigma}(\xi)\cap spt\mu_V}|q(x)|\le\delta^{\frac{1}{40}}\sigma,\\
   &iv) \ \ \mathcal{H}^2((graphf\backslash sptV)\cap B_{\sigma}(\xi))+\mu_V(B_{\sigma}(\xi)\backslash graphf)\le 2^{83}\delta^{\frac{1}{16}}\pi \sigma^2,
\end{align*}
where $q:\mathbb{R}^{2+k}\to T^\bot$ is the orthogonal projection.
\end{thm}
\subsection{Preliminaries}
\

We begin with some preliminaries: the monotonicity formula and its corollaries.
\begin{lemma}\cite{LS93}\cite{KS}
Assume $V=\underline{v}(\Sigma, \theta)$ is a rectifiable 2-varifold in an open set $U\subset \mathbb{R}^{2+k}$ with generalized mean curvature $H\in L^2(d\mu)$.  Then, for any $x\in \mathbb{R}^{2+k}$, and $0<\sigma<\rho<\infty$ with $B_\rho(x)\subset U$,
\begin{align}\label{monotonicity equality}
\frac{\mu(B_{\sigma}(x))}{\sigma^2}=
\frac{\mu(B_{\rho}(x))}{\rho^2}&+\frac{1}{16}\int_{B_{\rho}(x)\backslash B_{\sigma}(x)}|H|^2d\mu-\int_{B_{\rho}\backslash B_{\sigma}}|\frac{\nabla^\bot r}{r}+\frac{H}{4}|^2d\mu\nonumber\\
&+\frac{1}{2\rho^2}\int_{B_{\rho}(x)}r\langle\nabla^\bot r,H\rangle d\mu-\frac{1}{2\sigma^2}\int_{B_{\sigma}(x)}r\langle\nabla^\bot r,H\rangle d\mu,
\end{align}
Where $r=r_x=|\cdot-x|$. Moreover, for any $\delta\le 1$, we have
\begin{align}\label{monotonicity inequality}
\frac{\mu(B_{\sigma}(x))}{\sigma^2}\le
(1+\delta)\frac{\mu(B_{\rho}(x))}{\rho^2}+\frac{1}{2\delta}\int_{B_{\rho}(x)}|H|^2d\mu.
\end{align}
\end{lemma}
\begin{cor}\label{density at each point}
Assume $V=\underline{v}(\Sigma, \theta)$ is a rectifiable 2-varifold in an open set $U\subset \mathbb{R}^{2+k}$ with generalized mean curvature $H\in L^2(d\mu)$ and $B_\rho(0)\subset U$.  If $\int_{B_{\rho}(0)}|H|^2d\mu+\mu(B_{\rho}(0))<\infty$ and $\theta(x)\ge 1$, for $\mu-a.e. x\in spt\mu$. Then
\begin{align*}
\Theta(x)=\lim_{\tau\to 0}\frac{\mu(B_\tau(x))}{\pi\tau^2}
\end{align*}
is well-defined in $\breve{B}_\rho(0)$ and is upper semi-continuous. Moreover, for any $x\in \breve{B}_\rho(0)$,  $$\Theta(x)\ge 1.$$
\end{cor}
\begin{proof}
See \cite[Appendix]{KS}
\end{proof}

The following corollary is prepared for section \ref{application}. For simplicity, we will omit the measure notation $d\mu$ under the integral from now on.
\begin{cor}\label{integral discription of upper density at infinity}
Assume $V=\underline{v}(\Sigma, \theta)$ is a rectifiable 2-varifold in $\mathbb{R}^n$ with generalized mean curvature $H\in L^2(\mathbb{R}^n,d\mu_V)$. Then, for any $x\in \mathbb{R}^{n}$, $$\Theta^*(V,\infty):=\limsup_{r\to \infty}\frac{\mu_V(B_r(x))}{\pi r^2}<+\infty$$ if and only if $$\Theta_*(V,\infty):=\liminf_{r\to \infty}\frac{\mu_V(B_r(x))}{\pi r^2}<+\infty$$ if and only if $$\int_{\mathbb{R}^n}|\frac{\nabla^{\bot}r_x}{r_x}|^2<\infty.$$

Moreover, if one of the above condition holds, then  for any $\rho\in (0,\infty)$,
\begin{align}\label{everyradius}
\frac{\mu_V(B_\rho(x))}{\pi \rho^2}\le 9\Theta_*(V,\infty)+\frac{59}{16\pi}\int_{\mathbb{R}^n}|H|^2.
\end{align}
\end{cor}
\begin{proof}
For simplicity, we denote $r_x=|\cdot-x|$ by $r$. Since $$|\frac{1}{2\sigma^2}\int_{B_{\sigma}(x)}r\langle \nabla^{\bot}r,H\rangle|\le \frac{(\mu(B_{\sigma}(x)))^{\frac{1}{2}}}{2\sigma}\|H\|_{L^{2}(B_{\sigma}(x))}\to 0,$$
letting $\sigma\to 0$ in the monotonicity formula (\ref{monotonicity equality}), we get
\begin{align}
\int_{B_{\rho}(x)}|\frac{\nabla^{\bot}r}{r}+\frac{H}{4}|^2-\frac{1}{2\rho^2}\int_{B_{\rho}(x)}r\langle\nabla^{\bot}r,H\rangle
=\frac{\mu_V(B_{\rho}(x))}{\rho^2}-\pi \Theta(x)+\frac{1}{16}\int_{B_{\rho}(x)}|H|^2,
\end{align}
which implies $\int_{B_{\rho}(x)}|\frac{\nabla^{\bot}r}{r}|^2<+\infty$.
Note
\begin{align*}
\frac{1}{2}\int_{B_{\rho}(x)}|\frac{\nabla^{\bot}r}{r}|^2-\int_{B_{\rho}(x)}|\frac{H}{4}|^2\le\int_{B_{\rho}(x)}|\frac{\nabla^{\bot}r}{r}+\frac{H}{4}|^2\le2\int_{B_{\rho}(x)}(|\frac{\nabla^{\bot}r}{r}|^2+|\frac{H}{4}|^2)
\end{align*}
and
\begin{align*}
|\frac{1}{2\rho^2}\int_{B_{\rho}(x)}r\langle \nabla^{\bot}r,H\rangle|\le \frac{1}{4}\int_{B_{\rho}(x)}|\frac{\nabla^{\bot}r}{r}|^2+\frac{1}{4}\int_{B_{\rho}(x)}|H|^2.
\end{align*}
We know
\begin{align*}
\frac{1}{4}\int_{B_{\rho}(x)}|\frac{\nabla^{\bot}r}{r}|^2-\frac{3}{8}\int_{B_{\rho}(x)}|H|^2
&\le\frac{\mu_V(B_{\rho}(x))}{\rho^2}-\pi \Theta(x)\\
&\le  \frac{9}{4}\int_{B_{\rho}(x)}|\frac{\nabla^{\bot}r}{r}|^2+\frac{5}{16}\int_{B_{\rho}(x)}|H|^2.
\end{align*}
Letting $\rho\to\infty$, we get
\begin{align*}
\frac{1}{4}\int_{\mathbb{R}^n}|\frac{\nabla^{\bot}r}{r}|^2-\frac{3}{8}\int_{\mathbb{R}^n}|H|^2
&\le\pi(\Theta_*(V,\infty)-\Theta(x))\\
&\le\pi(\Theta^{*}(V,\infty)-\Theta(x))\\
&\le\frac{9}{4}\int_{\mathbb{R}^n}|\frac{\nabla^{\bot}r}{r}|^2+\frac{5}{16}\int_{\mathbb{R}^n}|H|^2.
\end{align*}
Finally, combining the last two lines we get
\begin{align*}
\frac{\mu(B_\rho(x)}{\pi \rho^2}
\le 9\Theta_*(V,\infty)+\frac{59}{16\pi}\int_{\mathbb{R}^n}|H|^2.
\end{align*}
\end{proof}

\subsection{Semi-Reifenberg Condition}
\

In the proof of the Allard regularity theorem, the following non-dimensional(scaling invariant) quantity $E(\xi, \rho, T)$ plays an important role.
\begin{defn}
Assume $V=\underline{v}(\Sigma, \theta)$ is a rectifiable 2-varifold in $\mathbb{R}^{2+k}$ and $B_\rho(\xi)\subset U$. Denote $\mu=\mu_V$. For any $2$-plane $T$ in $\mathbb{R}^{2+k}$, the tilt-Excess $E(\xi, \rho, T)$ is defined by
\begin{align*}
E(\xi,\rho,T):&=\rho^{-2}\int_{B_{\rho}(\xi)}|p_{T_x\Sigma-p_T}|^2d\mu,
\end{align*}
where $T_x\Sigma$ is the approximate tangent plane of the varifold $V$ at $x\in spt\mu$ and $p_T$ and $p_{T_x\Sigma}$ are orthogonal projection to $T$ and $T_x\Sigma$ respectively.
\end{defn}
 The tilt-excess measures the mean oscillation of the approximate tangent space( Gaussian map) of the varifold in the ball $B_\rho(\xi)$. The oscillation behavior of the tangent spaces are always relating to the regularity of the geometric objects at different levels. For example, the $C^{1,\alpha}$ regularity occurred in the Allard regularity theorem owes to the decay of tilt-excess. Stephen Semmes proved \cite{S91a}\cite{S91b}\cite{S91c} the Lipschitz regulairty for hypersurfaces in $\mathbb{R}^{n+1}$ with Gaussian maps small BMO. And Reifenberg's topological disk theorem, the key to the $C^\alpha$ regularity, is also established on some oscillation condition--the Reifenberg condition (\ref{reifenberg con}).
 \begin{thm}[\bf{Reifenberg}]\label{reifenberg theorem}\cite{R60}\cite{M66}\cite{LS96} For integers $m,k>0$ and $\alpha>0$, there exists $\varepsilon=\varepsilon(m,k,\alpha)>0$ such that for any closed set $S\subset \mathbb{R}^{m+k}$ with $0\in S$, if for any $y\in S\cap B_1(0)$ and $\rho\in (0,1]$, there exists an $m$-dimensional plane $L_{y,\rho}\subset \mathbb{R}^{m+k}$ passing through $y$ such that
 \begin{align}\label{reifenberg con}
 d_{\mathcal{H}}(S\cap B_\rho(y), L_{y,\rho}\cap B_{\rho}(y))\le \varepsilon \rho,
 \end{align}
 then $S\cap B_1$ is homeomorphic to the unit ball $B_1^m(0)\subset \mathbb{R}^m$. More precisely, there exist closed set $M\subset \mathbb{R}^{m+k}$ and $m$-dimensional subspace $T_0\subset \mathbb{R}^{m+k}$ and a homeomorphism $\tau: T_0\to M$ such that $M\cap B_1=S\cap B_1$, both $\tau,\tau^{-1}\in C^{\alpha}$ and
 \begin{align*}
 |\tau(x)-x|\le C(m,k)\varepsilon, \forall x\in T_0 \ \ \ \text{ and } \ \ \ \ \tau(x)=x, \forall x\in T_0\backslash B_2.
 \end{align*}
 \end{thm}
The condition (\ref{reifenberg con}) is called the Reifenberg condition. In this subsection, we establish the tilt-excess estimate.  By the way, we note the process in fact implies half of the Reifenberg condition, we call it semi-Reifenberg condition (\ref{semireifenberg}).

   By noting the integrand in the tilt-excess is just the gradient of the position function, the tilt-excess estimate can be reduced to some ``$L^2$-estimate" by integral gradient estimate of the generalized mean curvature equation.
\begin{lemma}\label{integral gradient estimate}
Assume $V=\underline{v}(\Sigma, \theta)$ is a rectifiable 2-varifold in an open set $U\subset \mathbb{R}^{2+k}$ with generalized mean curvature $H\in L^2(d\mu)$ for $\mu=\mu_V$ and $B_\rho(\xi)\subset U$. Then, for any  $2$-plane $T$ in $\mathbb{R}^{2+k}$,
$$E(\xi,\frac{\rho}{2},T)\le 4\int_{B_{\rho}(\xi)}|H|^2+592\rho^{-2}\int_{B_{\rho}(\xi)}(\frac{d(x,T)}{\rho})^2d\mu.$$
\end{lemma}
\begin{proof}
Take coordinates of $\mathbb{R}^{2+k}$ such that $T=\text{span}\{(x^1,x^2,0,\ldots, 0)\}$ and $T^{\bot}=\text{span}\{X'=(0,0,x^3,x^4,...x^{2+k})\}$. Then, an observation is
\begin{align}\label{gradient expression}
\frac{1}{2}|p_{T_x\Sigma-p_T}|^2=\Sigma_{j=1}^{k}|\nabla^{\Sigma}x^{2+j}|^2=div^{\Sigma}X'.
\end{align}
So, to estimate $E(\xi,\rho,T)=2\rho^{-2}\int_{B_{\rho}(\xi)}\Sigma_{j=1}^{k}|\nabla^{\Sigma}x^{2+j}|^2d\mu$ is equal to give an integral gradient estimate of the generalized mean curvature equation
$$\int div^{\Sigma}X=-\int X\cdot\overrightarrow{H}.$$
For details, see \cite[Lemma 22.2]{LS83}
\end{proof}
The above lemma reduces the tilt-excess estimate to the ``$L^2$- estimate" of the form  $\rho^{-2}\int_{B_\rho(\xi)}(\frac{d(x,T)}{\rho})^2d\mu$, whose estimate can be seen as an integral version of half of the Reifenberg condition.  The following lemma gives a point-wise semi-Reifenberg condition, which implies the tilt-excess estimate.

\begin{lemma}[$\mathbf{Semi-Reifenberg\ Condition}$]\label{Semi-Reifenberg Condition}
Assume $V=\underline{v}(\Sigma,\theta)$ is a rectifiable 2-varifold in $U\supset B_{\rho}(0)\subset \mathbb{R}^{2+k}$ with generalized mean curvature $H\in L^2(d\mu)$,  $0\in sptV$ and $\theta\ge 1$ for $\mu_V-a.e. x\in U$. If for some $\delta\le2^{-4}$,
 $$\frac{\mu_V{B_{\rho}(0)}}{\pi \rho^2}\le1+\delta\text{ and }\int_{B_{\rho}(0)}|H|^2\le \delta,$$  then for $\forall \xi\in spt\mu_V\cap B_{\frac{1}{2}\delta^{\frac{1}{2}}\rho}(0)$ and $\forall \sigma\le\frac{1}{2}\delta^{\frac{1}{2}}\rho$, there exists a $T=T(\xi,\sigma)$ passing through $\xi$, such that
\begin{align}\label{semireifenberg}
\sigma^{-1}\sup_{x\in spt\mu_V\cap B_{\sigma}(\xi)}d(x,T)\le 2^{13}\delta^{1/16}.
\end{align}
\end{lemma}
\begin{proof}
Step 3.1\
 Volume ratio estimate. For $\forall \xi\in B_{\delta^{\frac{1}{2}}\rho}(0)$ and $\forall \sigma\in (0,(1-\delta^{\frac{1}{2}})\rho)$, we have
\begin{align}\label{upperbound}
\frac{\mu(B_{\sigma}(\xi))}{\pi\sigma^2}\le1+36\delta^{\frac{1}{2}}.
\end{align}
Moreover, if $\xi\in spt\mu_V$, then
\begin{align}\label{lowerbound}
\frac{\mu(B_{\sigma}(\xi))}{\pi\sigma^2}\ge1-2\delta^{\frac{1}{2}}.
\end{align}
In fact, take $\beta=\delta^{\frac{1}{2}}\le\frac{1}{2}$ and $\delta_0=\delta^{\frac{1}{2}}$. Then by the monotonicity formula (\ref{monotonicity inequality}), we know
\begin{align*}
\frac{\mu(B_{\sigma}(\xi))}{\pi\sigma^2}\le(1+\delta_0)\frac{\mu(B_{\rho-\beta\rho})(\xi)}{\pi(\rho-\beta\rho)^2}+\frac{1}{2\pi\delta_0}\int_{B_{\rho-\beta\rho}(\xi)}|H|^2\le1+36\delta^{\frac{1}{2}}.
\end{align*}
On the other hand, for $\xi\in spt\mu_V$, (\ref{monotonicity inequality}) and Corollary \ref{density at each point} imply
\begin{align*}
1\le(1+\delta_0)\frac{\mu(B_{\sigma}(\xi))}{\pi\sigma^2}+\frac{1}{2\pi\delta_0}\int_{B_{\sigma}(\xi)}|H|^2
\le(1+\delta_0)\frac{\mu(B_{\sigma}(\xi))}{\pi\sigma^2}+\frac{\delta}{2\delta_0}.
\end{align*}
Thus
\begin{align*}
\frac{\mu(B_{\sigma}(\xi))}{\pi\sigma^2}\ge\frac{1-\frac{\delta^{\frac{1}{2}}}{2}}{1+\delta^{\frac{1}{2}}}
\ge1-2\delta^{\frac{1}{2}}.
\end{align*}
Step 3.2 For $\forall \xi\in spt\mu_V\cap B_{\beta\rho}(0)$ and small $\sigma$, the goal is to find $T=T(\xi,\sigma)$ such that
$$\sigma^{-1}\sup\{d(x,T):d\in spt\mu_V\cap B_{\sigma}(\xi)\} \text{ small }.$$
It is not easy to get the point estimate directly.  So, in the spirit of Chebyshev inequality, we estimate the mean integral value in a small neighborhood of $\xi$. More precisely, for small $\alpha$(to be determined) and  $y \in spt\mu_V\cap B_{\alpha\sigma}(\xi)$, denote $T_y$ to be the translation of the approximate tangent space of $\Sigma$ at $y$(which exists for $\mu$-almost $y$ since $V$ is rectifiable) such that $T_y\ni y$. Then $d(x,T_y)=|p^{\bot}_{T_y}(x-y)|$ measures how close $x\in spt\mu_V\cap B_{\sigma}(\xi)$ is to $T_y$.  Consider the mean integral
\begin{align*}
\bar{J}&=\frac{1}{(\alpha\sigma)^2}\int_{B_{\alpha\sigma}(\xi)}\frac{1}{\sigma^2}\int_{B_{\sigma}(\xi)}\frac{d^2(x,T_y)}{\sigma^2}d\mu(x)d\mu(y)\\
&\le\frac{1}{\sigma^2}\int_{B_{\sigma}(\xi)}\frac{1}{(\alpha\sigma)^2}\int_{B_{(\alpha+1)\sigma}(x)}\frac{d^2(x,T_y)}{\sigma^2}d\mu(y)d\mu(x).
\end{align*}
For fixed $x$ and $r_x(y)=|y-x|$, note
\begin{align}\label{remainder formula}
|\nabla^{\bot}r_x(y)|=r_x^{-1}(y)|p_{T_y}^{\bot}(y-x)|=r_x^{-1}(y)d(x,T_y).
\end{align}
 So
 \begin{align*}
 \bar{J}&\le\frac{1}{\sigma^2}\int_{B_{\sigma}(\xi)}\frac{1}{(\alpha\sigma)^2}\int_{B_{(\alpha+1)\sigma}(x)}\frac{r_x^4}{\sigma^2}\frac{|\nabla^{\bot}r_x|^2}{r_x^2}d\mu(y)d\mu(x)\\
 &=\frac{(1+\alpha)^4}{\alpha^2\sigma^2}\int_{B_{\sigma}(\xi)}{\blue\underline{\int_{B_{(\alpha+1)\sigma}(x)}\frac{|\nabla^{\bot}r_x|^2}{r_x^2}d\mu(y)}_{=:K(x)}}d\mu(x)\\
 \end{align*}
 Note  $\Theta(x)\ge1$ for  $x\in spt\mu_V$ and
 \begin{align*}
 \lim_{\sigma_1\to0}|\frac{1}{\sigma_1^2}\int_{B_{\sigma_1}(x)}r_x\langle\nabla^{\bot}r_x,H\rangle|
 \le \lim_{\sigma_1\to0}(\int_{B_{\sigma_1}(x)}|H|^2)^{\frac{1}{2}}(\frac{\mu(B_{\sigma_1}(x))}{\sigma_1^2})^{\frac{1}{2}}
 =0.
 \end{align*}
  Taking $\rho_1=(1+\alpha)\sigma$, using the monotonicity formula (\ref{monotonicity equality}) for $0<\sigma_1<\rho_1$ and then letting $\sigma_1\to 0$,  we get
 \begin{align*}
\int_{B_{\rho_1}(x)}|\frac{\nabla^{\bot}r}{r}+\frac{H}{4}|^2
&\le(\frac{\mu(B_{\rho_1}(x))}{\rho_1^2}-\pi)+\frac{1}{16}\int_{B_{\rho_1}(x)}|H|^2+\frac{1}{2\rho_1}\int_{B_{\rho_1}(x)}|H|.
\end{align*}
Taking $\xi\in spt\mu_V\cap B_{\frac{1}{2}\delta^{\frac{1}{2}}\rho}(0)$, $\alpha\le 1$ and $\sigma\le\frac{1}{2}\delta^{\frac{1}{2}}\rho$, then $x\in B_{\sigma}(\xi)\in B_{\delta^{\frac{1}{2}}\rho}(0)$
 and $(1+\alpha)\sigma<(1-\delta^{\frac{1}{2}})\rho$. So, by (\ref{upperbound}), we get
\begin{align}\label{remainder term}
K(x)
&\le\int_{B_{\rho_1}(x)}\frac{|\nabla^{\bot}r|^2}{r^2}d\mu(y)
    \le2\int_{B_{\rho_1}(x)}|\frac{\nabla^{\bot}r}{r}+\frac{H}{4}|^2
    +2\int_{B_{\rho_1}(x)}|\frac{H}{4}|^2\nonumber\\
&\le2[(\frac{\mu(B_{\rho_1}(x))}{\rho_1^2}-\pi)+\frac{2}{16}\int_{B_{\rho_1}(x)}|H|^2
    +\frac{1}{2}(\int_{B_{\rho_1}(x)}|H|^2)(\frac{\mu(B_{\rho_1}(x))}{\rho_1})^{\frac{1}{2}}]\nonumber\\
&\le80\pi\delta^{\frac{1}{2}},
\end{align}
and
\begin{align*}
\bar{J}\le\frac{2^4 \cdot80\pi\delta^{\frac{1}{2}}\mu(B_{\sigma}(\xi))}{\alpha^2\sigma^2}\le \frac{2^{12}\pi\delta^{\frac{1}{2}}}{\alpha^2}=2^{12}\pi\delta^{\frac{1}{4}}
({\blue\text{ take } \alpha=\delta^{\frac{1}{8}}}).
\end{align*}
Thus by Chebyshev's inequality,  there exists $y\in B_{\alpha\sigma}(\xi)$, such that
$$\frac{1}{\sigma^2}\int_{B_{\sigma}(\xi)}\frac{d^2(x,T_y)}{\sigma^2}d\mu(x)\le2^{12}\delta^{\frac{1}{4}}/(1-2\delta^{\frac{1}{2}})\le2^{13}\delta^{\frac{1}{4}}{ \blue(\text{ take }\delta\le2^{-4})}.$$
So if we denote $T=T_y+\xi-y$, then $d^2(x,T)\le2d^2(x,T_y)+2{\alpha}^2\sigma^2$. Thus
\begin{align}\label{integral semi reifenberg}
\frac{1}{\sigma^2}\int_{B_{\sigma}(\xi)}\frac{d^2(x,T)}{\sigma^2}d\mu(x)\le2^{14}\delta^{\frac{1}{4}}
+2\delta^{\frac{1}{4}}(\pi+32\pi\delta^{\frac{1}{2}})\le2^{15}\delta^{\frac{1}{4}}.
\end{align}
And by Lemma \ref{integral gradient estimate}, we know that
\begin{align*}
E(\xi,\frac{\sigma}{2},T(\xi,\sigma))\le 4\delta+592\cdot 2^{15}\delta^{\frac{1}{4}}\le2^{25}\delta^{\frac{1}{4}}.
\end{align*}
Up to now, we have established the tilt-excess estimate:
\begin{cor}\label{Tilt-Excess estimate}
under the condition of Lemma \ref{Semi-Reifenberg Condition}, for $\delta\le2^{-4}$, $ \xi\in B_{\frac{1}{2}\delta^{\frac{1}{2}}\rho}(0)$ and $ \sigma\le\frac{1}{2}\delta^{\frac{1}{2}}\rho$, there exists $T=T(\xi,\sigma)$ such that
$$E(\xi,\sigma,T)\le 2^{25}\delta^{1/4}.$$
\end{cor}
We write the corollary to emphasize this is enough for the Lipschitz approximation Theorem \ref{Lipschitz Approximation for 2-varifold}(see proof below). But for the final goal of the $C^{\alpha}$-rugularity, the integral semi-Reifenberg condition (\ref{integral semi reifenberg}) is not enough, the point-wise estimate (\ref{semireifenberg}) is necessary. We follow the argument as in \cite[Section 24]{LS83} to complete the proof.
\bigskip

Step 3.2$'$
 By (\ref{remainder term}) and (\ref{remainder formula}), for $\delta\le2^{-4},\alpha\le1$, $\forall \xi\in spt\mu_V\cap B_{\frac{1}{2}\delta^{\frac{1}{2}}\rho}(0), \forall \sigma\le\frac{1}{2}\delta^{\frac{1}{2}}\rho$ and $\forall x\in spt\mu_V\cap B_{\sigma}(\xi)$, we have
\begin{align*}
I(x):&=\int_{B_{\alpha\sigma}(\xi)}|p_{T_y}^{\bot}(x-y)|^2d\mu(y)=\int_{B_{\alpha\sigma}(\xi)}d^2(x,T_y)d\mu(y)\\
     &\le\int_{B_{(1+\alpha)\sigma}(x)}r_x^2(y)|\nabla^{\bot}r_x|^2(y)d\mu(y)\le[(1+\alpha)\sigma]^4K(x)\le2^{11}\pi\sigma^4\delta^{\frac{1}{2}}.
\end{align*}
Now, take a maximal disjoint collection $\{B_{\frac{\alpha\sigma}{4}}(x_i)\}_{i=1}^N$ of balls with radius $\frac{\alpha\sigma}{4}$ and centered in $spt\mu_V\cap B_{\sigma}(\xi)$. Then we have
\begin{align*}
spt\mu_V\cap B_{\sigma}(\xi)\subset \cup_{i=1}^NB_{\alpha\sigma}(x_i)
\end{align*}
and
\begin{align*}
\mu_V(B_{\frac{\alpha\sigma}{4}}(x_i))\ge\pi(\frac{\alpha\sigma}{4})^2(1-2\delta^{\frac{1}{2}}).
\end{align*}
Moreover, we know
\begin{align*}
N\le\frac{\mu_V(B_{\sigma}(\xi))}{\pi(\frac{\alpha\sigma}{4})^2(1-2\delta^{\frac{1}{2}})}\le 2^5\frac{\mu_V(B_{\sigma}(\xi))}{\pi\sigma^2\alpha^2}\le\frac{2^8}{\alpha^2},
\end{align*}
and
\begin{align*}
\int_{B_{\alpha\sigma}(\xi)}\Sigma_{i=1}^{N}|p^{\bot}_{T_y\Sigma}(x_i-y)|^2d\mu(y)
\le\frac{2^8}{\alpha^2}\int_{B_{\alpha\sigma}(\xi)}|p^{\bot}_{T_y}(x_i-y)|^2d\mu(y)\le\frac{2^{19}\pi\sigma^4\delta^{\frac{1}{2}}}{\alpha^2}.
\end{align*}
Take $\alpha=\delta^{\frac{1}{16}}$ and $\theta=2^{21}\sigma^2\delta^{\frac{1}{4}}$. Then we get
\begin{align*}
\frac{\mu_V(B_{\alpha\sigma}(\xi)\cap\{\Sigma_{i=1}^N|p^{\bot}_{T_y}(x_i-y)|^2\ge\theta\})}{\mu_V(B_{\alpha\sigma}(\xi))}
\le\frac{\frac{2^{19}\pi\sigma^4\delta^{\frac{1}{2}}}{\alpha^2\theta}}{(1-2\delta^{\frac{1}{2}})\pi(\alpha\sigma)^2}\le\frac{1}{2}<1.
\end{align*}
Thus there exists a point $y_0\in B_{\delta^{\frac{1}{16}}\sigma}(\xi)$ such that
\begin{align*}
\Sigma_{i=1}^N|p^{\bot}_{T_{y_0}}(x_i-y_0)|^2<\theta=2^{21}\sigma^2\delta^{\frac{1}{4}}.
\end{align*}
That is,
\begin{align*}
\sup_{1\le i\le N}|p^{\bot}_{T_{y_0}}(x_i-y_0)|\le 2^{11}\sigma\delta^{\frac{1}{8}}.
\end{align*}
So, for $\forall x\in spt\mu_V\cap B_{\sigma}(\xi)\subset\cup_{i=1}^N B_{\alpha\sigma}(x_i)=\cup_{i=1}^N B_{\delta^{\frac{1}{16}}\sigma}(x_i)$, there exists $1\le i\le N$ such that $x\in B_{\delta^{\frac{1}{16}}\sigma}(x_i)$. Thus
\begin{align*}
|p^{\bot}_{T_{y_0}}(x-y_0)|\le|p^{\bot}_{T_{y_0}}(x-x_i)|+|p^{\bot}_{T_{y_0}}(x_i-y_0)|\le \delta^{\frac{1}{16}}\sigma+2^{11}\sigma\delta^{\frac{1}{8}}\le2^{12}\delta^{\frac{1}{16}}\sigma.
\end{align*}
So, if we let $T=T_{y_0}+\xi-y_0$, then $\xi\in T$ and for any $x\in spt\mu_V\cap B_{\sigma}(\xi)$,
\begin{align*}
\sigma^{-1}d(x,T)=\sigma^{-1}|p^{\bot}_T(x-\xi)|\le\sigma^{-1}(|p^{\bot}_T(x-y_0)|+|y_0-\xi|)\le 2^{13}\delta^{\frac{1}{16}}.
\end{align*}
\end{proof}
\subsection{Lipschitz Approximation}
\

Firstly, we need the following version of weighted monotonicity inequality, which roughly means most of the measure concentrate in the neighborhood of a plane.
\begin{lemma}\label{band}
Assume $V=\underline{v}(\Sigma,\theta)$ is a rectifiable 2-varifold in $U$ with generalized mean curvature $H\in L^2(\mu_V)$. Then for $l\in(0,1),\beta\in(0,1/4),B_R(\xi)\subset U $ and  $\forall y\in B_{\beta R}(\xi)$, we have
\begin{align*}
\pi \Theta(\mu_V,y)\le&(1+24\beta)\frac{\mu_V\big(\{x:|q_0(x-y)|<2l\beta R\}\cap B_{R}(\xi)\big)}{R^2}\\
&+\frac{6}{(l\beta)^5}\frac{1}{R^2}\int_{B_{R}(\xi)}\|p_{T_x\Sigma}-p_{\mathbb{R}^2\times\{0\}}\|^2+\frac{2}{(l\beta)^3}\int_{B_{R}(\xi)}|H|^2,
\end{align*}
where $q_0$ is the orthogonal projection of $\mathbb{R}^{2+k}$ onto $\{0\}\times\mathbb{R}^k$.
\end{lemma}
\begin{proof}
 W.l.o.g., assume $\xi=0$ and $y\in B_{\beta R}(0)$. Denote $T_y=\mathbb{R}^2\times \{0\}+y$ and define $q_y:\mathbb{R}^{2+k}\to T_y^{\bot}$, $q_y(y+(x_1,x_2))=y+x_2$, $p_y=id-q_y$ to be the orthogonal projection to $T_y^{\bot}$ and $T_y$ respectively. For $\alpha$ to be determined, choose a function $g\in C^1(\mathbb{R},[0,1])$ such that $g(t)\equiv 1$ for $t\in[-\alpha R,\alpha R]$, $g(t)\equiv 0$ for $|t|\ge2\alpha R$ and $|g'(t)|\le \frac{2}{\alpha R}$. Put $h(x)=g(|q_y(x)-y|)$. We will deduce a monotonicity formula involving the weight $h^2$. Take $X=h^2\eta r\nabla r$( here $r(x)=r_y(x)=|x-y|, \eta=\eta(r)$ to be determined below). Since
\begin{align*}
div^{\Sigma}X
&=\langle\nabla^{\Sigma}(h^2\eta), r\nabla r\rangle+h^2\eta div^{\Sigma}(r\nabla r)\\
&=h^2r\eta'+2h^2\eta+2\eta hr\langle\nabla^{\Sigma}r,
\nabla^{\Sigma}h\rangle-h^2r\eta'|\nabla^{\bot}r|^2,
\end{align*}
by the definition of generalized mean curvature $\int div^{\Sigma}X=-\int X\cdot H$, we know
\begin{align*}
LH:=\int h^2(r\eta'+2\eta)=\int h^2r\eta'|\nabla^{\bot}r|^2-2\eta hr\langle\nabla^{\Sigma}r,\nabla^{\Sigma}h\rangle-h^2\eta r\langle\nabla^{\bot}r, H\rangle=:RH.
\end{align*}
As before, since $B_{(1-\beta)R}(y)\subset B_R(0)$, for $0<\sigma<\rho<(1-\beta)R$, we can take
$$
    \eta(r)=
    \begin{cases}
      f(\sigma)-f(\rho),& r\le \sigma,\\
      f(r)-f(\rho),&r\in(\sigma,\rho).\\
    \end{cases}
$$
for decreasing $f$ to be chosen. Then
$$
    \eta'(r)=
    \begin{cases}
      0,& r\le \sigma,\\
      f'(r),&r\in(\sigma,\rho),\\
    \end{cases}
$$
\begin{align*}
LH&=\int_{B_{\rho}\backslash B_{\sigma}(y)}h^2rf'(r)+2\int_{B_{\sigma}(y)}h^2(f(\sigma)-f(\rho))+2\int_{B_{\rho}\backslash B_{\sigma}(y)}h^2(f(r)-f(\rho))\\
&=\int_{B_{\rho}\backslash B_{\sigma}(y)}h^2(rf'+2f)+2f(\sigma)\int_{B_{\sigma}(y)}h^2-2f(\rho)\int_{B_{\rho}(y)}h^2\\
\end{align*}
and
$RH=-2T+RH_1$
for
$T=\int_{B_{\rho}(y)}\eta hr\langle\nabla^{\Sigma}r, \nabla^{\Sigma}h\rangle$
and
\begin{align*}
RH_1&=\int_{B_{\rho}\backslash B_{\sigma}(y)}h^2\{rf'|\nabla^{\bot}r|^2-rf\langle\nabla^{\bot}r, H\rangle\}+f(\rho)\int_{B_{\rho}\backslash B_{\sigma}(y)}h^2r\langle\nabla^{\bot}r, H\rangle\\
& \ \ \ \ \  -(f(\sigma)-f(\rho))\int_{B_{\sigma}(y)}h^2r\langle\nabla^{\bot}, H\rangle\\
&=-\int_{B_{\rho}\backslash B_{\sigma}(y)}h^2|\sqrt{-rf'}\nabla^{\bot}r+\frac{rfH}{2\sqrt{-rf'}}|^2+\int_{B_{\rho}\backslash B_{\sigma}(y)}h^2|\frac{rfH}{2\sqrt{-rf'}}|^2\\
& \ \ \ \ \ +f(\rho)\int_{B_{\rho}(y)}h^2r\langle\nabla^{\bot}r,H\rangle-f(\sigma)\int_{B_{\sigma}(y)}h^2r\langle\nabla^{\bot}r,H\rangle.
\end{align*}
Especially, if we take $f(r)=\frac{1}{r^2}$, then $rf'+2f=0$, $\sqrt{-rf'}=\sqrt{2}r^{-1}$, $rf=r^{-1}$ and $\frac{rf}{2\sqrt{-rf'}}=\frac{1}{2\sqrt{2}}$, then by $LH=RH$ and $-\int_{B_{\rho}\backslash B_{\sigma}(y)}|\frac{\nabla^{\bot}r}{r}+\frac{H}{4}|^2h^2\le 0$, we get
\begin{align}\label{L<H}
\frac{1}{\sigma^2}\int_{B_{\sigma}(y)}h^2-\frac{1}{\rho^2}\int_{B_{\rho}(y)}h^2\le-T+T_4-T_5+\frac{1}{16}\int_{B_{\rho}\backslash B_{\sigma}(y)}|H|^2h^2,
\end{align}
where
\begin{align*}
T_4=\frac{1}{2\rho^2}\int_{B_{\rho}(y)}h^2r\langle\nabla^{\bot}r,H\rangle,\ \  T_5=\frac{1}{2\sigma^2}\int_{B_{\sigma}(y)}h^2r\langle\nabla^{\bot}r,H\rangle
\end{align*}
and
\begin{align*}
T&=\int_{B_{\rho}\backslash B_{\sigma}(y)}(f(r)-f(\rho))h\langle r\nabla^{\Sigma}r,\nabla^{\Sigma} h\rangle+\int_{B_{\sigma}(y)}(f(\sigma)-f(\rho))h\langle r\nabla^{\Sigma}r,\nabla^{\Sigma} h\rangle\\
&=\int_{B_{\rho}\backslash B_{\sigma}(y)}h\langle \frac{\nabla^{\Sigma}r}{r},\nabla ^{\Sigma}h\rangle+\frac{1}{\sigma^2}\int_{B_{\sigma}(y)}h\langle r\nabla^{\Sigma}r,\nabla^{\Sigma} h\rangle-\frac{1}{\rho^2}\int_{B_{\rho}(y)}h\langle r\nabla^{\Sigma}r,\nabla^{\Sigma} h\rangle\\
&=:T_1+T_2+T_3.
\end{align*}
By Young's inequality, we know
\begin{align*}
|T_2|\le\frac{\varepsilon}{\sigma^2}\int_{B_{\sigma}(y)}h^2+\frac{1}{4\varepsilon}\int_{B_{\sigma}(y)}|\nabla^{\Sigma}h|^2,
\ \ \ \
|T_3|\le\frac{\varepsilon}{\rho^2}\int_{B_{\rho}(y)}h^2+\frac{1}{4\varepsilon}\int_{B_{\rho}(y)}|\nabla^{\Sigma}h|^2,
\end{align*}
\begin{align*}
|T_4|\le\frac{\varepsilon}{\rho^2}\int_{B_{\rho}(y)}h^2+\frac{1}{8\varepsilon}\int_{B_{\rho}(y)}|H|^2h^2,
\ \ \ \
|T_5|\le\frac{\varepsilon}{\sigma^2}\int_{B_{\sigma}(y)}h^2+\frac{1}{8\varepsilon}\int_{B_{\sigma}(y)}|H|^2h^2,
\end{align*}
and
\begin{align*}
|T_1|\le\int_{B_{\rho}\backslash B_{\sigma}(y)}\frac{\varepsilon h^2}{r^2}+\frac{|\nabla^{\Sigma}h|^2}{4\varepsilon}
\le\varepsilon\frac{\rho^2}{\sigma^2}\frac{1}{\rho^2}\int_{B_{\rho}(y)}h^2+\frac{1}{4\varepsilon}\int_{B_{\rho}\backslash B_{\sigma}(y)}|\nabla^{\Sigma}h|^2.
\end{align*}
Substitute the estimate of $T_i,i=1,2,3,4,5$ into (\ref{L<H}). We get, for $\varepsilon<\frac{1}{2}$,
\begin{align*}
(1-2\varepsilon)\frac{1}{\sigma^2}\int_{B_{\sigma}(y)}h^2\le&(1+2\varepsilon+\varepsilon(\frac{\rho}{\sigma})^2)\frac{1}{\rho^2}\int_{B_{\rho}(y)}h^2\\
&+\frac{3}{4\varepsilon}\int_{B_{\rho}(y)}|\nabla^{\Sigma}h|^2+(\frac{1}{16}+\frac{1}{4\varepsilon})\int_{B_{\rho}(y)}|H|^2h^2.
\end{align*}
On the one hand, by definition, we know for $x\in B_{\alpha R}(y)$, $|q_y(x)-y|=|q_0(x-y)|\le\alpha R$. Hence $h\equiv 1$ on $B_{\alpha R}(y)$. So if we take $\sigma=\alpha R<\rho<(1-\beta)R$ and use the monotonicity inequality (\ref{monotonicity inequality}), then
\begin{align*}
\pi\Theta(\mu_V,y)
\le&(1+\varepsilon)\frac{\mu_V(B_{\alpha R}(y))}{(\alpha R)^2}+\frac{1}{2\varepsilon}\int_{B_{\alpha R}(y)}|H|^2\\
\le&(1+\varepsilon)\frac{1}{\sigma^2}\int_{B_{\sigma}(y)}h^2+\frac{1}{2\varepsilon}\int_{B_{\rho}(y)}|H|^2\\
\le&\frac{1+\varepsilon}{1-2\varepsilon}\{(1+2\varepsilon+\varepsilon(\frac{\rho}{\sigma})^2)\frac{1}{\rho^2}\int_{B_{\rho}(y)}h^2
+\frac{1}{\varepsilon}\int_{B_{\rho}(y)}(|\nabla^{\Sigma}h|^2+|H|^2)\},
\end{align*}
where we use $h\le1$ in the last inequality.

On the other hand, since $\nabla^{\mathbb{R}^{2+k}}|q_y(x)-y|=\frac{q_0(x-y)}{|q_0(x-y)|}$, we know for $x\in spt\mu_V$ where $T_x\Sigma$ exists,
\begin{align*}
|\nabla^{\Sigma}h|^2(x)\le&(|g'(q_0(x-y))||\nabla^{\Sigma}|q_0(x-y)||)^2\\
\le&\big(\frac{2}{\alpha R}\big)^2\big|\frac{p_{T_x\Sigma}(q_0(x-y))}{|q_0(x-y)|}\big|^2\\
\le&\frac{4}{(\alpha R)^2}\|p_{T_x\Sigma}-p_{\mathbb{R}^2\times\{0\}}\|^2.
\end{align*}
Moreover, $sptq\subset[-2\alpha R,2\alpha R]$ implies $spt h\subset \{x:|q_0(x-y)|<2\alpha R\}$. Thus if we take $\rho=(1-2\beta)R$, then $\frac{\rho}{\sigma}=\frac{1-2\beta}{\alpha}$, $B_{\rho}(y)\subset B_{R}(\xi)$ and
\begin{align*}
\pi\Theta(\mu_V,y)
\le&\frac{1+\varepsilon}{1-2\varepsilon}\{(1+2\varepsilon+\varepsilon(\frac{\rho}{\sigma})^2)\frac{1}{\rho^2}\int_{B_{\rho}(y)}h^2
+\frac{1}{\varepsilon}\int_{B_{\rho}(y)}(|\nabla^{\Sigma}h|^2+|H|^2)\}\\
\le&\frac{1+\varepsilon}{1-2\varepsilon}\{(1+2\varepsilon+\frac{(1-2\beta)^2\varepsilon}{\alpha^2})\frac{\mu_V\big(\{x:|q_0(x-y)|<2\alpha R\}\cap B_{R}(\xi)\big)}{((1-2\beta)R)^2}\\
&\ \ \ +\frac{4}{\varepsilon(\alpha R)^2}\int_{B_{(1-2\beta)R}(y)}\|p_{T_x\Sigma}-p_{\mathbb{R}^2\times\{0\}}\|^2+\frac{1}{\varepsilon}\int_{B_{(1-2\beta)R}(y)}|H|^2\}.
\end{align*}
Take $\alpha=l\beta$ and $\varepsilon=(l\beta)^3=\alpha^3\le 2^{-6}$ for $l<1$. Then
\begin{align*}
\frac{1+\varepsilon}{1-2\varepsilon}\le 1+6(l\beta)^3
\ \ \ \text{ and }\ \ \
1+2\varepsilon+\frac{(1-2\beta)^2\varepsilon}{\alpha^2}\le 1+2l\beta.
\end{align*}
Thus
\begin{align*}
\pi\Theta(\mu_V,y)
\le&(1+6(l\beta)^3)(1+2l\beta)\frac{\mu_V\big(\{x:|q_0(x-y)|<2l\beta R\}\cap B_{(1-2\beta)R}(y)\big)}{((1-2\beta)R)^2}\\
&+\frac{6}{(l\beta)^5}\frac{1}{R^2}\int_{B_{(1-2\beta)R}(y)}\|p_{T_x\Sigma}-p_{\mathbb{R}^2\times\{0\}}\|^2+\frac{2}{(l\beta)^3}\int_{B_{(1-2\beta)R}(y)}|H|^2\\
\le&(1+24\beta)\frac{\mu_V\big(\{x:|q_0(x-y)|<2l\beta R\}\cap B_{R}(0)\big)}{R^2}\\
&+\frac{6}{(l\beta)^5}E(0,R,T)+\frac{2}{(l\beta)^3}\int_{B_{R}(0)}|H|^2.
\end{align*}
\end{proof}

\begin{cor}\label{graph condition}
 Assume $\alpha, l\in (0,1)$ and $V=\underline{v}(\Sigma,\theta)$ is a rectifiable 2-varifold in $U$ with generalized mean curvature $H\in L^2(\mu_V)$ and  $\frac{\mu_V(B_R(\xi))}{\pi R^2}\le 2-\alpha$.  For $\beta_1=\beta_1(\alpha)=\frac{\alpha}{48(2-\alpha)}$ and $\delta=\delta(\alpha)=\frac{\alpha^3}{2^{23}}$, if
$$l^{-5}E(\xi,R,T)\le \delta^2\pi \ \ and\ \ l^{-3}\int_{B_R(\xi)}|H|^2\le \delta^2\pi, $$
then for $\forall y,z\in B_{\beta_1 R}(\xi)$ with $|y-z|\ge \beta_1 R$, we have
$$|q_0(y-z)|\le l|y-z|.$$
\end{cor}
\begin{proof}
We assume $T=\mathbb{R}^2\times \{0\}$ and argue by contradiction. Otherwise,  there exist $y,z\in spt\mu_V\cap B_{\beta_1 R}(\xi)$ with $|y-z|\ge \beta_1 R$ but $|q_0(y-z)|>l|y-z|$. Thus for $\forall x\in B_{R}(\xi)$,
\begin{align*}
|q_0(x-y)|+|q_0(x-z)|\ge |q_0(y-z)|> l|y-z|\ge l\beta_1 R,
\end{align*}
So, either $|q_0(x-y)|>\frac{l\beta_1 R}{2}$ or $|q_0(x-z)|>\frac{l\beta_1 R}{2}$, i.e.,
\begin{align*}
\{x\in B_{R}(\xi):|q_0(x-y)|\le\frac{l\beta_1 R}{2}\}\cap\{x\in B_{R}(\xi):|q_0(x-z)|\le\frac{l\beta_1 R}{2}\}=\emptyset.
\end{align*}
Noting $y,z\in spt\mu_V\cap B_{\beta_1 R}(\xi)$, by Lemma \ref{band}, we get
\begin{align*}
2\pi\le& (\Theta(\mu_V,y)+\Theta(\mu_V,z))\pi\\
\le&(1+24\beta_1)\frac{\mu_V\big(\{x\in B_{R}(\xi):|q_0(x-z)|<\frac{1}{2}l\beta_1 R\text{ or } |q_0(x-z)|<\frac{1}{2}l\beta_1 R\} \big)}{R^2}\\
&\ \ \ +\frac{12}{(\frac{1}{2}l\beta_1)^5}\frac{1}{R^2}\int_{B_{R}(\xi)}\|p_{T_x\Sigma}-p_{\mathbb{R}^2\times\{0\}}\|^2
+\frac{4}{(\frac{1}{2}l\beta_1)^3}\int_{B_{R}(\xi)}|H|^2\\
\le&(1+24\beta_1)\frac{\mu_V(B_R(\xi))}{R^2}+\frac{12}{(\frac{1}{2}l\beta_1)^5}E(\xi,R,\mathbb{R}^2\times\{0\})
+\frac{4}{(\frac{1}{2}l\beta_1)^3}\int_{B_{R}(\xi)}|H|^2\\
\le& (1+24\beta_1)(2-\alpha)\pi+\frac{3\cdot2^{7}}{l^5\beta_1^5}l^5\delta^2\pi+\frac{2^5}{l^3\beta_1^3}l^3\delta^2\pi\\
\le& (2-\frac{\alpha}{8})\pi.
\end{align*}
 A contradiction!
\end{proof}

\begin{prop}[Lipschitz Approximation 0.5 version]\label{Lipschitz Approximation 0.5 version}
For $\forall \alpha\in (0,1)$, there exists
\begin{align*}
&\beta_3(\alpha)=\frac{1}{4}\beta_2(\alpha)=\frac{\alpha}{2^8\cdot3\cdot5(4-\alpha)},\ \  \beta_4(\alpha)=\frac{\alpha^3}{2^{21}}\\
&\delta^2_3(\alpha)=\delta_2^4\delta_1^{10}=\frac{\alpha^{12}}{2^{174}k^5}{\ \blue (\delta_1=\frac{1}{2^7k},\delta_2(\alpha)=\frac{\alpha^3}{2^{26}})}
\end{align*}
such that the following statement holds:

 Assume $V=\underline{v}(\Sigma,\theta)$ is a rectifiable 2-varifold in $U\subset \mathbb{R}^{2+k}$ satisfying
  \begin{align*}
   &(1) \ \ 0\in spt\mu_V, B_R(0)\subset U,\\
   &(2) \ \ \frac{\mu_V(B_R(0))}{\pi R^2}\le 2-\alpha,\\
   &(3)\ \ \theta\ge1, \mu_V-a.e. x\in U.
\end{align*}
For any $l\in(0,1)$, if
\begin{align*}
l^{-5}E:=l^{-5}E(0,R,\mathbb{R}^2\times \{0\})\le \delta_3^2 \text{ and } l^{-3}W:=l^{-3}\int_{B_R(0)}H^2\le \delta_3^2,
 \end{align*}
 then for any $\beta\in (\beta_4,\beta_3)$, there exists a Lipschitz function $f=(f^1,f^2, \ldots, f^k):B_{\beta R}^{\mathbb{R}^2\times\{0\}}(0)\to \mathbb{R}^k$ with
$$Lipf\le l,\ \ \sup_{x\in B_{\beta R}}|f|\le l\beta R $$
and for $F=Graphf$,
\begin{align*}
\mathcal{H}^2((F\backslash spt\mu_V)\cap B_{\beta R}(0))+\mu_V(B_{\beta R}(0)\backslash F)\le 2^{27}(l^{-\frac{5}{2}} E^{\frac{1}{2}}+l^{-\frac{3}{2}} W^{\frac{1}{2}})\pi R^2.
\end{align*}
Moreover, for the orthogonal projection $q_0:\mathbb{R}^{2+k}\to \{0\}\times\mathbb{R}^k$, we have
\begin{align*}
\sup_{x\in B_{\beta R}(0)\cap spt\mu_V}|q_0(x)|\le l\beta R.
\end{align*}
\end{prop}
\begin{proof}
Following the notation of Corollary \ref{graph condition}, take
$$\delta_0(\alpha)=\beta_0(\alpha)=\frac{\alpha}{40},\ \ \ \ \ \ \ \ \ \ \ \delta_1\in(0,1)\text{ to be determined },$$
$$\beta_2(\alpha)=\frac{1}{20}\beta_1(\frac{\alpha}{2})=\frac{\alpha}{960(4-\alpha)},\ \ \ \ \ \ \  \delta_2(\alpha)=\delta(\frac{\alpha}{2})=\frac{\alpha^3}{2^{26}},\ \ \ \ \ \ \  \delta_3^2=\delta_2^4\delta_1^{10}.$$
The same as the proof of (\ref{upperbound}) and (\ref{lowerbound}), by the monotonicity formula, it is easy to show that for $x\in B_{\beta_0 R}(\xi)\cap spt\mu_V$ and $\sigma\in (0,(1-\beta_0) R)$, we have
 \begin{align}\label{uplowerbound}
 1-2\delta_0\le\frac{\mu_V(B_{\sigma}(x))}{\pi \sigma^2}\le 2-\frac{\alpha}{2}.
 \end{align}

 Letting
 \begin{align*}
G:=\{x\in spt\mu_V\cap B_{\beta R}(0):\frac{E(x,\sigma,\mathbb{R}^2\times \{0\})}{(l\delta_1)^{5}}+\frac{\int_{B_{\sigma}(x)}|H|^2}{(l\delta_1)^{3}}\le\pi\delta_2^2,\forall \sigma<\frac{R}{10} \},
\end{align*}
then for any  $\beta\in (\beta_4,\beta_2)$, $ x\in G \text{ and } y\in spt\mu_V\cap B_{\beta R}(0)$, we have
$$\sigma:=\frac{|x-y|}{\beta_1(\frac{\alpha}{2})}<\frac{2\beta_2 R}{\beta_1(\frac{\alpha}{2})}\le \frac{R}{10}\le (1-\beta_0)R\ \ \text{ and }\ \  x\in B_{\beta R}(0)\subset B_{\beta_0 R}(0).$$
By (\ref{uplowerbound}) we get
\begin{align*}
\frac{\mu_V\Big(B_{\frac{|x-y|}{\beta_1(\frac{\alpha}{2})}}(x)\Big)}{\pi\Big(\frac{|x-y|}{\beta_1(\frac{\alpha}{2})}\Big)^2}\le 2-\frac{\alpha}{2}.
\end{align*}
Since $x\in G$, we know that for $\sigma=\frac{|x-y|}{\beta_1(\frac{\alpha}{2})}$,
$$\frac{E(x,\sigma,\mathbb{R}^2\times \{0\})}{(l\delta_1)^{5}}+\frac{\int_{B_{\sigma}(x)}|H|^2}{(l\delta_1)^{3}}\le\pi\delta^2(\frac{\alpha}{2}).$$
Thus by Corollary \ref{graph condition} we know
\begin{align}\label{lip estimate}
|q_0(x-y)|\le l\delta_1|x-y|.
\end{align}
Especially, if we take $y=0\in spt\mu_V\cap B_{\beta R}(0)$, then
\begin{align}\label{Hausdorff distance estimate}
\sup_{x\in G}|q_0(x)|\le \delta_1 l|x|\le\delta_1 l\beta R.
\end{align}
Moreover, if we define $p_0:\mathbb{R}^{2+k}\to \mathbb{R}^2$, $p_0(x_1,x_2)=x_1$ and $\Omega_0=p_0(G)$, then for $x,y\in G$, by (\ref{lip estimate}) we know
$$|q_0(x-y)|\le \frac{\delta_1 l}{\sqrt{1-(\delta_1 l)^2}}|p_0(x-y)|\le\frac{2\delta_1l}{\sqrt{3}}|p_0(x-y)|{\red (\text{ for } \delta_1\le \frac{1}{2}}).$$

Thus
 $$G=Graph f_0,\ \ \  f_0(p_0(x))=q_0(x):\Omega_0\to \mathbb{R}^k, \ \ Lipf_0\le \frac{2\delta_1l}{\sqrt{3}}.$$
 Now by the extension theorem of Lipschitz function, there exists a Lipschitz function $\tilde{f}:\mathbb{R}^2\to \mathbb{R}^k$ such that $\tilde{f}=f$ on $\Omega_0$ and $Lip\tilde{f}\le kLipf_0\le \frac{2k\delta_1l}{\sqrt{3}}$. Noting that $|f_0|\le \sup_{G}|q_0|\le \delta_1 l \beta R$, we can put $f=\min\{\max\{f,-l\beta R\},l\beta R\}$ and get an extending $f$ of $f_0$ with
 \begin{align}\label{lipconstant}
 Lipf\le \frac{2k\delta_1l}{\sqrt{3}}\ \ \ \text{ and }\ \ \ sup|f|\le l\beta R.
 \end{align}
   Noting $spt\mu_V\backslash Graphf\subset spt\mu_V\backslash G$, we estimate $\mu_V\big((spt\mu_V\backslash G)\cap B_{\beta R}(0)\big)$ next.
 For any $x\in (spt\mu_V\backslash G)\cap B_{\beta R}(0)$, there is $\sigma_{x}\in (0,\frac{R}{10})$ such that
 \begin{align*}
 (l\delta_1)^{-5}E(x,\sigma_{x},\mathbb{R}^2\times \{0\})+(l\delta_1)^{-3}\int_{B_{\sigma_x}(x)}|H|^2\ge\pi\delta_2^2.
 \end{align*}
 Let
 \begin{align*}
 A:=&\{x\in spt\mu_V\cap B_{\beta R}(0):(l\delta_1)^{-5}E(x,\sigma_x,\mathbb{R}^2\times \{0\})\ge\frac{1}{2}\pi\delta_2^2\},\\
 B:=&\{x\in spt\mu_V\cap B_{\beta R}(0):(l\delta_1)^{-3}\int_{B_{\sigma_x}(x)}|H|^2\ge\frac{1}{2}\pi\delta_2^2\}.
 \end{align*}
 Then $(spt\mu_V\backslash G)\cap B_{\beta R}(0)\subset A\cup B$.

 By $5$-times lemma, there exists disjoint collection $\{B_{\sigma_{x_j}}(x_j)\}_{j=1}^{\infty}$ of $\{B_{\sigma_x}(x)\}_{x\in A}$ such that $A\subset \cup_{j=1}^{\infty} B_{5\sigma_{x_j}}(x_j)$. For $x\in A$, we know
  $$\sigma_x^2\le\frac{2}{\delta_2^2(\delta_1 l)^5}\int_{B_{\sigma_x}(x)}\|_{T_x\Sigma}-p_{\mathbb{R}^2\times\{0\}}\|^2d\mu_V(x).$$
  Since $5\sigma_{x_j}\le \frac{5 R}{10}\le (1-\beta_0) R$, by (\ref{uplowerbound}) we know $\frac{\mu_V(B_{5\sigma_{x_j}}(x_j))}{\pi (5\sigma_{x_j})^2}\le 2-\frac{\alpha}{2}.$
  Thus
  \begin{align*}
  \mu(A)&\le \Sigma_{j=1}^{\infty}\mu(B_{5\sigma_{x_j}}(x_j))\le \Sigma_{j=1}^{\infty}(2-\frac{\alpha}{2})25\pi \sigma_{x_j}^2\\
        &\le 50\pi \Sigma_{j=1}^{\infty}\frac{2}{\delta_2^2(\delta_1 l)^5}\int_{B_{\sigma_{x_j}}(x_j)}\|p_{T_x\Sigma}-p_{\mathbb{R}^2\times\{0\}}\|^2d\mu\\
        &\le \frac{100\pi}{\delta_2^2(\delta_1 l)^5}\int_{B_R(0)}\|p_{T_x\Sigma}-p_{\mathbb{R}^2\times\{0\}}\|^2d\mu\\
        &\le 100\pi l^{-\frac{5}{2}} R^2 E^{\frac{1}{2}}(0,R,\mathbb{R}^2\times \{0\}),
  \end{align*}
  where we use the condition $l^{-5}E(0,R,\mathbb{R}^2\times \{0\})\le \delta_2^4 \delta_1^{10}$  in the last inequality.
  Similarly, there exists disjoint collection $\{B_{\sigma_{y_j}}(y_j)\}_{j=1}^{N}$ of $\{B_{\sigma_y}(y)\}_{y\in B}$ such that $B\subset \cup_{j=1}^{N} B_{5\sigma_{y_j}}(y_j)$. Since $\int_{B_{\sigma_y}(y)}|H|^2\ge\frac{1}{2}\pi\delta_2^2(l\delta_1)^{3}$ for any $y\in B$, we know  $N\le \frac{\int_{B_R(0)}|H|^2}{\delta_2^2(l\delta_1)^{3}}$ and
  \begin{align*}
  \mu(B)&\le \Sigma_{j=1}^{N}\mu(B_{5\sigma_{y_j}}(y_j))\le \Sigma_{j=1}^{N}(2-\frac{\alpha}{2})25\pi \sigma_{y_j}^2\\
        &\le \frac{\pi}{2}N R^2\le \frac{\pi R^2}{2\delta_2^2(\delta_1 l)^3}\int_{B_R(0)}|H|^2\\
        &\le \pi l^{-\frac{3}{2}} R^2 (\int_{B_R(0)}|H|^2)^{\frac{1}{2}},
  \end{align*}
  where we use $l^{-3}\int_{B_R(0)}|H|^2\le \delta_2^4 \delta_1^{6}$ in the last line.
  As a result,
  \begin{align}\label{sptoutgraph}
  \mu_V((spt\mu_V\backslash G)\cap B_{\beta R}(0))&\le \pi R^2 l^{-\frac{5}{2}} [100E^{\frac{1}{2}}(0,R,\mathbb{R}^2\times \{0\})+l (\int_{B_R(0)}|H|^2)^{\frac{1}{2}}]\\
  &\le101\pi R^2 \delta_2^2\delta_1^3\le 101\frac{\alpha^6}{2^{52}}\pi R^2\nonumber\\
  {\blue (\text{ since } \beta>\frac{\alpha^3}{2^{21}})}&<\frac{1}{2}\pi(\beta R)^2\nonumber\\
  {\blue(\text{ by } 0\in spt\mu_V\text{ and }(\ref{uplowerbound}))}&< \mu_V(B_{\beta R}(0)).\nonumber
   \end{align}
   So $G\neq \emptyset$. Taking $ x_0\in G$, by (\ref{Hausdorff distance estimate}) and (\ref{lip estimate}) we know $|q_0(x_0)|\le \delta_1l\beta R$  and
   \begin{align}\label{Hausdorff distance estimate outside graph}
\sup_{y\in spt\mu_V\cap B_{\beta R}(0)}|q_0(y)|\le \sup_{y\in spt\mu_V\cap B_{\beta R}(0)}|q_0(y-x_0)|+|q_0(x_0)| \le3\delta_1 l\beta R.
\end{align}

 Next, we estimate the $\mathcal{H}^2\big((Graph f\backslash spt\mu_V)\cap B_{\frac{\beta}{4} R}(0)\big)$. For this, set $F=Graph f$ and denote
 \begin{align*}
  C:=(F\backslash spt\mu_V)\cap B_{\frac{\beta}{4} R}(0).
 \end{align*}
  For $\forall \eta\in C$, take $\sigma_{\eta}$ to be the smallest $\sigma$ such that $B_{\frac{\sigma}{2}}(\eta)\cap spt\mu_V=\emptyset$ but $B_{\frac{3\sigma}{4}}(\eta)\cap spt\mu_V\neq\emptyset$. Since $\eta\notin spt\mu_V$, $\sigma_{\eta}>0$ and $0\in spt\mu_V$, we know $\frac{\sigma_{\eta}}{2}\le |\eta|\le \frac{\beta}{4} R$ and $\sigma_{\eta}\le \frac{\beta R}{2}$. Now, $B_{\frac{3\sigma_{\eta}}{4}}(\eta)\cap spt\mu_V\neq\emptyset$ implies there is $ \xi_{\eta}\in spt\mu_V\cap B_{\frac{3\sigma_{\eta}}{4}}(\eta)\subset spt\mu_V\cap B_{\frac{3\beta R}{8}}(\eta)\subset spt\mu_V\cap B_{\beta R}(0)$. Thus $B_{\sigma_{\eta}}(\eta)\supset B_{\frac{1}{4}\sigma_{\eta}}(\xi_{\eta})$ and by (\ref{uplowerbound}) we know,
\begin{align}\label{onehand}
\mu(B_{\sigma_{\eta}}(\eta))\ge \mu(B_{\frac{1}{4}\sigma_{\eta}}(\xi_{\eta}))\ge (1-2\delta_0)\pi (\frac{1}{4}\sigma_{\eta})^2.
\end{align}
On the other hand, since $B_{\frac{\sigma_{\eta}}{2}}(\eta)\cap spt\mu=\emptyset$, by the monotonicity formula (\ref{monotonicity equality}) we know
\begin{align}\label{otherhand}
\int_{B_{\sigma_{\eta}}(\eta)}|\frac{\nabla^{\bot}r}{r}+\frac{H}{4}|^2&=\frac{\mu(B_{\sigma_{\eta}}(\eta))}{\sigma_{\eta}^2}
+\frac{1}{16}\int_{B_{\sigma_{\eta}}(\eta)}|H|^2+\frac{1}{2\sigma_{\eta}^2}\int_{B_{\sigma_{\eta}}(\eta)}r\langle \nabla^{\bot}r,H\rangle\nonumber\\
 &\ge (1-\varepsilon)\frac{\mu(B_{\sigma_{\eta}}(\eta))}{\sigma_{\eta}^2}+(\frac{1}{16}-\frac{1}{4\varepsilon})\int_{B_{\sigma_{\eta}}(\eta)}|H|^2.
\end{align}
Taking $\varepsilon=\frac{1}{2}$ in (\ref{otherhand}) and using (\ref{onehand}), we get
\begin{align}\label{long}
\frac{(1-2\delta_0)\pi\sigma_{\eta}^2}{16}\le\mu(B_{\sigma_{\eta}}(\eta))
&\le \Big(2\int_{B_{\sigma_{\eta}}(\eta)}|\frac{\nabla^{\bot}r}{r}+\frac{H}{4}|^2+\frac{7}{8}\int_{B_{\sigma_{\eta}}(\eta)}|H|^2\Big)\sigma_\eta^2\nonumber\\
&\le \Big(4\int_{B_{\sigma_{\eta}}(\eta)}|\frac{\nabla^{\bot}r}{r}|^2+\frac{9}{8}\int_{B_{\sigma_{\eta}}(\eta)}|H|^2\Big)\sigma_{\eta}^2\nonumber\\
{\blue (By\  (\ref{remainder formula}))}
&\le\Big(4\int_{B_{\sigma_{\eta}}(\eta)}|\frac{p_{T_x\Sigma}^{\bot}(x-\eta)}{|x-\eta|^2}|^2+2\int_{B_{\sigma_{\eta}}(\eta)}|H|^2\Big)\sigma_{\eta}^2\nonumber\\
{\blue(Since\ spt\mu\cap B_{\frac{\sigma_{\eta}}{2}}(\eta)=\emptyset)}&\le\frac{4\sigma_{\eta}^2}{(\frac{\sigma_{\eta}}{2})^2}
\int_{B_{\sigma_{\eta}}(\eta)}|p_{T_x\Sigma}^{\bot}\big(\frac{x-\eta}{|x-\eta|}\big)|^2+\underline{2\sigma_{\eta}^2\int_{B_{\sigma_{\eta}}(\eta)}|H|^2}\nonumber\\
{\blue(-2\sigma_{\eta}^2\int_{B_{\sigma_{\eta}}(\eta)}|H|^2)}
&\lesssim\underline{32\int_{B_{\sigma_{\eta}}(\eta)}\|p_{T_x\Sigma}^{\bot}-q_0\|^2}+32\int_{B_{\sigma_{\eta}}(\eta)}|q_0\big(\frac{x-\eta}{|x-\eta|}\big)|^2\nonumber\\
{\blue(-32\int_{B_{\sigma_{\eta}}(\eta)}\|p_{T_x\Sigma}-p_0\|^2)}
&\lesssim 32\mu(B_{\sigma_{\eta}}(\eta)\backslash F)+32\int_{B_{\sigma_{\eta}}(\eta)\cap F}|q_0\big(\frac{x-\eta}{|x-\eta|}\big)|^2\nonumber\\
{(\blue By\  Lipf\le \frac{2k\delta_1l}{\sqrt{3}})}
&\le 32\mu(B_{\sigma_{\eta}}(\eta)\backslash F)+32(\frac{2k\delta_1l}{\sqrt{3}})^2\mu(B_{\sigma_{\eta}}(\eta))\nonumber\\
&\le 32\mu(B_{\sigma_{\eta}}(\eta)\backslash F)+2^6(k\delta_1l)^2(2-\frac{\alpha}{2})\sigma_{\eta}^2,
\end{align}
where we use the non-standard notation
\begin{align*}
&A_1+\underline{S_1}\\
(-S_1)\lesssim &A_2+\underline{S_2}\\
(-S_2)\lesssim & A_3
\end{align*}
to mean $A_1\le A_2+S_1\le A_3+S_1+S_2$ to save space.
{\red Fixing $\delta_1=\frac{1}{2^7k}$}, then
\begin{align}\label{short}
2\sigma_{\eta}^2\int_{B_{\sigma_{\eta}}(\eta)}|H|^2+2^6(k\delta_1l)^2(2-\frac{\alpha}{2})\sigma_{\eta}^2
\le \big(2l^3\delta_2^4\delta_1^6+2^6(k\delta_1l)^2(2-\frac{\alpha}{2})\big)\sigma_{\eta}^2\le 2^{-6}\sigma_{\eta}^2.
\end{align}
Noticing $\frac{(1-2\delta_0)\pi\sigma_{\eta}^2}{16}\ge \frac{\sigma_{\eta}^2}{2^5}$ and substituting (\ref{short}) into (\ref{long}), then we get
\begin{align*}
\frac{\sigma_\eta^2}{2^5}\le 32(\int_{B_{\sigma_{\eta}}(\eta)}\|p_{T_x\Sigma}-p_0\|^2+\mu(B_{\sigma_{\eta}}(\eta)\backslash F))+2^{-6}\sigma_{\eta}^2,
\end{align*}
i.e.,
\begin{align*}
\sigma_{\eta}^2\le2^{11}\big(\int_{B_{\sigma_{\eta}}(\eta)}\|p_{T_x\Sigma}-p_0\|^2+\mu(B_{\sigma_{\eta}}(\eta)\backslash F)\big).
\end{align*}
 Denote $\eta'=p_0(\eta)$, then the Lebesgue measure
 \begin{align*}
 \mathcal{L}^2(B_{5\sigma_{\eta}}^{\mathbb{R}^2\times\{0\}}(\eta'))=\pi(5\sigma_{\eta})^2\le2^{11}\cdot 5^2\pi\big(\int_{B_{\sigma_{\eta}}(\eta)}\|p_{T_x\Sigma}-p_0\|^2+\mu(B_{\sigma_{\eta}}(\eta)\backslash F)\big).
 \end{align*}
 Again by the $5$-times lemma, there exists disjoint collection $\{B^{\mathbb{R}^2\times\{0\}}_{\sigma_{\eta_j}}(\eta'_j)\}_{j=1}^{\infty}$ of $\{B^{\mathbb{R}^2\times\{0\}}_{\sigma_{\eta}}(\eta')\}_{\eta\in C}$ such that $p_0(C)\subset \cup_{j=1}^{\infty} B^{\mathbb{R}^2\times\{0\}}_{5\sigma_{\eta_j}}(\eta'_j)$. Thus
 $$\cup_{j=1}^{\infty}B_{\sigma_{\eta_j}}(\eta_j)\subset \cup_{j=1}^{\infty}(B_{\sigma_{\eta_j}}^{\mathbb{R}^2\times\{0\}}(\eta'_j)\times\mathbb{R}^k)\cap B_{\beta R}(0)\subset B_{\beta R}(0)$$
 and
 \begin{align*}
 \mathcal{L}^2(p_0(C))\le \Sigma_{j=1}^{\infty}\mathcal{L}^2(B_{5\sigma_{\eta_j}}^{\mathbb{R}^2\times\{0\}}(\eta'_j))
 \le 2^{18}\big(\int_{B_{\beta R}(0)}\|p_{T_x\Sigma}-p_0\|^2+\mu(B_{\beta R}(0)\backslash F)\big).
 \end{align*}
 Moreover, since $C\subset Graph f$ for some $f$ with $Lipf\le 1$ and $(spt\mu_V\backslash F)\cap B_{\beta R}$ is included in $(spt\mu_V\backslash G)\cap B_{\beta R}$ whose measure has been estimated, we know
  \begin{align}\label{graphoutspt}
 \mathcal{H}^2&((F\backslash spt\mu)\cap B_{\frac{\beta}{4} R}(0))\nonumber\\
 &\le(\sqrt{2})^2\mathcal{L}^2(p_0(C))\nonumber\\
 &\le 2^{19}[\int_{B_{R}(0)}\|p_{T_x\Sigma}-p_0\|^2+\pi R^2 l^{-\frac{5}{2}} \big(100E^{\frac{1}{2}}(0,R,\mathbb{R}^2\times \{0\})+l (\int_{B_R(0)}H^2)^{\frac{1}{2}}\big)]\nonumber\\
 &\le2^{19}[101\pi R^2 l^{-\frac{5}{2}} \big(E^{\frac{1}{2}}(0,R,\mathbb{R}^2\times \{0\})+l (\int_{B_R(0)}H^2)^{\frac{1}{2}}\big)]\nonumber\\
 &\le2^{26}\pi R^2 \big(l^{-\frac{5}{2}} E^{\frac{1}{2}}(0,R,\mathbb{R}^2\times \{0\})+l^{-\frac{3}{2}} (\int_{B_R(0)}H^2)^{\frac{1}{2}}\big).
 \end{align}
 As a result, if we take $\beta_3(\alpha)=\frac{1}{4}\beta_2(\alpha)$,then for
 $\beta\in(\beta_4,\beta_3)$,by(\ref{lipconstant}) (\ref{sptoutgraph})(\ref{graphoutspt}) and (\ref{Hausdorff distance estimate outside graph}), we are done.
\end{proof}

Combing this theorem with the tilt-excess estimate( Corollary \ref{Tilt-Excess estimate}), we can finish the proof of Theorem \ref{Lipschitz Approximation for 2-varifold}.
\begin{proof}[ proof of Theorem \ref{Lipschitz Approximation for 2-varifold}]
 By (\ref{upperbound}) and  Corollary \ref{Tilt-Excess estimate} we know, for any $\xi \in spt\mu_V\cap B_R(\xi)$ and $R<\frac{1}{2}\delta^{\frac{1}{2}}\rho$, there exists a plane $T=T(\xi,R)$ such that $$\frac{\mu(B_{R}(\xi))}{\pi R^2}\le 1+36\delta^{\frac{1}{2}}=:2-\alpha,\ \ \ \ \text{ and } \ \ \ \ \ \ E(\xi,R,T)\le 2^{25}\delta^{\frac{1}{4}}.$$
 Since $\alpha=1-36\delta^{\frac{1}{2}}\in [\frac{1}{2},1]${\red \ (for $\delta\le \frac{1}{2^{16}}$)}, we know $\beta_4(\alpha)\le\frac{1}{2^{21}}\le\frac{1}{2^{14}}\le\beta_3(\alpha)$ and $\delta_3^2\ge \frac{1}{2^{186}k^5}$. So {\blue if
  \begin{align}\label{assumption}
  l^{-20}\delta\le \delta_5:= \frac{1}{2^{844}k^{20}},
  \end{align}}
  then \begin{align}\label{condition}
l^{-5}E(\xi,R,T)\le \delta_3^2\text{ and } l^{-3}W \le \delta_3^2(\text{ by  } \int_{B_R(\xi)}|H|^2\le \delta).
\end{align}
Thus by Proposition \ref{Lipschitz Approximation 0.5 version} we know, for any $\beta\in (\frac{1}{2^{21}},\frac{1}{2^{14}})$, there exists a Lipschitz function $f=(f^1,f^2, \ldots, f^k):B_{\beta R}^{T}(\xi)\to T^{\bot}$ with
$$Lipf\le l,\ \ \sup_{x\in B_{\beta R}(\xi)}|f|\le l\beta R $$
and for $F=Graphf$,
\begin{align*}
\mathcal{H}^2((F\backslash spt\mu_V)\cap B_{\beta R}(\xi))&+\mu_V(B_{\beta R}(\xi)\backslash F)\\
&\le 2^{27}(l^{-\frac{5}{2}} E^{\frac{1}{2}}+l^{-\frac{3}{2}} W^{\frac{1}{2}})\pi R^2\\
&\le 2^{27} (l^{-\frac{5}{2}} 2^{13}\delta^{\frac{1}{8}}+l^{-\frac{3}{2}} \delta^{\frac{1}{2}})\pi R^2\\
&\le 2^{83}l^{-\frac{5}{2}}\delta^{\frac{1}{8}}\pi (\beta R)^2.
\end{align*}
Moreover, for $q:\mathbb{R}^{2+k}\to T^{\bot}$ the orthogonal projection, we have
\begin{align*}
\sup_{x\in B_{\beta R}(\xi)\cap spt\mu_V}|q(x)|\le l\beta R.
\end{align*}
Especially, for $\delta\le \delta_6=\delta_5^2=\frac{1}{2^{1688}k^{40}}$, we can take $l=\delta^{\frac{1}{40}}$ such that (\ref{assumption}) holds. So, if we fix $\beta=\frac{1}{2^{15}}$ and denote $\sigma=\beta R$. Then we actually proved: for $\forall \xi\in B_{{\frac{\delta^{\frac{1}{2}}}{2^{16}}\rho}}(0)$ and $\forall \sigma\in(0,{\frac{\delta^{\frac{1}{2}}}{2^{16}}\rho})$, there exist a plane $T=T(\xi,\sigma)$ and a vector valued Lipschitz function $f=(f^1,f^2, \ldots, f^k):B_{\sigma}^{T}(\xi)\to T^{\bot}$ with
$$Lipf\le \delta^{\frac{1}{40}},\ \ \sup_{x\in B_{\sigma}(\xi)}|f|\le \delta^{\frac{1}{40}}\sigma $$
and for $F=Graphf$,
\begin{align*}
\mathcal{H}^2((F\backslash spt\mu_V)\cap B_{\sigma}(\xi))+\mu_V(B_{\sigma}(\xi)\backslash F)\le 2^{83}\delta^{\frac{1}{16}}\pi \sigma^2.
\end{align*}
Moreover, for $q:\mathbb{R}^{2+k}\to T^{\bot}$ the orthogonal projection, we have
\begin{align*}
\sup_{x\in B_{\sigma}(\xi)\cap spt\mu_V}|q(x)|\le \delta^{\frac{1}{40}}\sigma.
\end{align*}
\end{proof}
\section{$C^\alpha$-Regularity}\label{reifenberg}
In this section, we combine the Lipschitz approximation theorem and Reifenberg's topological theorem to finish the proof of the $C^{\alpha}$-regularity Theorem.   We have proved half of it in (\ref{semireifenberg}). As it is noted in the last section, to show the Lipschitz approximation Theorem \ref{Lipschitz Approximation for 2-varifold}, the integral semi-Reifenberg condition(Corollary \ref{Tilt-Excess estimate}) is enough. We will show that Theorem \ref{Lipschitz Approximation for 2-varifold} can feed back to provide another half of the Reifenberg condition. They together complete the proof of the $C^{\alpha}$-regularity.

\begin{thm}[$\mathbf{Allard-Reifenberg\ Type\ Regularity}$]\label{Holder Regularity}
Assume $V=\underline{v}(\Sigma,\theta)$ is a rectifiable 2-varifold in $U\supset B_{\rho}(0)\subset \mathbb{R}^{2+k}$ with $0\in sptV$ and $\theta\ge 1$ $\mu-a.e.x\in U$ for $\mu:=\mu_V:=\mathcal{H}^2\llcorner \theta$. Then there exists small $\delta'_6(=\frac{1}{2^{3536}k^{80}})$ such that for any $\delta\le \delta'_6$ if
$$\frac{\mu(B_{\rho}(0))}{\pi \rho^2}\le1+\delta\text{\  and \ }\int_{B_{\rho}(0)}|H|^2\le \delta,$$
then for any $\xi\in B_{\frac{1}{2}\delta^{\frac{1}{2}}\rho}(0)$ and  $\sigma\in (0,\frac{\delta^{\frac{1}{2}}}{2^{18}}\rho)$, there exists a plane $T=T(\xi, \sigma)$ passing through $\xi$ such that
\begin{align}\label{reifenberg condition}
 \sigma^{-1}d_{\mathcal{H}}(spt\mu_V\cap B_{\sigma}(\xi), T\cap B_{\sigma}(\xi))\le 2^{44}\delta^{\frac{1}{80}},
 \end{align}
 Where $d_{\mathcal{H}}$ is the Hausdorff distance in $\mathbb{R}^{2+k}$.

 Moreover, for any $\alpha\in (0,1)$ and $\varepsilon=\varepsilon(k,\alpha)$  the small constant in Reifenberg's topological  disk Theorem \ref{reifenberg theorem}, if $2^{44}\delta^{\frac{1}{80}}\le \varepsilon$, then $spt\mu_V\cap B_{\frac{1}{2^{19}}\delta^{\frac{1}{2}}\rho}(0)$ is $C^{\alpha}$ homeomorphic to  a $2$-dimensional  topological closed disk.
\end{thm}
\begin{proof}
 For $\xi\in B_{\frac{1}{2}\delta^{\frac{1}{2}}\rho}(0)$ and $\sigma<\frac{1}{2^{18}}\delta^{\frac{1}{2}\rho}$, consider the ball $B_{4\sigma}(\xi)$. By Lemma \ref{Semi-Reifenberg Condition}, there exists a plane $T$ passing through $\xi$ such that for any $x\in 
 spt\mu_V\cap B_{4\sigma}(\xi)$,
 \begin{align}\label{semi}
 \sigma^{-1}d(x,T)\le 2^{15}\delta^{\frac{1}{16}}.
 \end{align}
 For the same $T$, by Lemma \ref{integral gradient estimate} we know
\begin{align}\label{tilt}
E(\xi,2\sigma,T)
\le 4\int_{B_{4\sigma}(\xi)}|H|^2+592\cdot (4\sigma)^{-2}\int_{B_{4\sigma}(\xi)}(\frac{d(x,T)}{4\sigma})^2
\le2^{37}\delta^{\frac{1}{8}}.
\end{align}
Replace Corollary \ref{Tilt-Excess estimate} by (\ref{tilt}) in the proof of Theorem \ref{Lipschitz Approximation for 2-varifold}, we know that for $\delta'_5=\frac{1}{2^{1768}k^{40}}, l=\delta^{\frac{1}{80}}, \delta\le\delta'_6:={\delta'_5}^2, \beta=\frac{1}{2^{15}}$ and $2\sigma=\beta R$, there exists a Lipschitz function $f=(f^1,f^2,...f^k):B_{2\sigma}(\xi)\cap T\to \mathbb{R}^k:=T^{\bot}$ with
 $$Lipf\le\delta^{\frac{1}{80}},\ \ \operatorname*{sup}\limits_{x'\in B_{2\sigma}(\xi)\cap T}|f(x')|\le\delta^{\frac{1}{80}}\cdot2\sigma$$
  and
\begin{align}\label{smallmeasure}
   \mathcal{H}^2((graphf\backslash sptV)\cap B_{2\sigma}(\xi))+\mu_V(B_{2\sigma}(\xi)\backslash graphf)\le 2^{76}\delta^{\frac{1}{32}}\pi (2\sigma)^2.
\end{align}
 Now, for any $x'\in B_{\sigma}(\xi)\cap T$, denote $x=(x',f(x'))$ and define $d(x)=\min\{d(x,spt\mu_V\cap B_{\sigma}(\xi)),\frac{1}{2}\sigma\}$. Then for any $y'\in B_{\frac{d(x)}{4}}(x')\cap B_{(1-2\delta^{\frac{1}{80}})\sigma}(\xi)\cap T$ and $y=(y',f(y'))$, we have
 \begin{align*}
 d(y,x)\le \sqrt{1+(Lipf)^2}d(x',y')\le \sqrt{1+\delta^{\frac{1}{40}}}\frac{d(x)}{4}\le \frac{d(x)}{2}
 \end{align*}
 and
 \begin{align*}
 d(y,\xi)\le |y'-\xi|+|f(y')|\le (1-2\delta^{\frac{1}{80}})\sigma+\delta^{\frac{1}{80}}(2\sigma)=\sigma.
 \end{align*}
 Thus
 \begin{align}\label{largedomain}
 B_{\frac{d(x)}{4}}(x')\cap B_{(1-2\delta^{\frac{1}{80}})\sigma}(\xi)\cap T\subset p(B_{\frac{d(x)}{2}}(x)\cap  B_{\sigma}(\xi)),
 \end{align}
 where $p:\mathbb{R}^{2+k}\to T$ is the orthogonal projection. We now claim
 $$d(x)\le 2^{43}\delta^{\frac{1}{80}}\sigma.$$

 To see this, we assume $d(x)\ge 16\delta^{\frac{1}{80}}\sigma$ without loss of generality.

  In the case $d(x',\partial B_{(1-2\delta^{\frac{1}{80}})\sigma}(\xi)\cap T)\le2\delta^{\frac{1}{80}}\sigma\le \frac{d(x)}{8}$, there exists a point $x''\in \partial B_{(1-2\delta^{\frac{1}{80}})\sigma}(\xi)\cap T$ such that $B_{\frac{d(x)}{4}}(x')\supset B_{\frac{d(x)}{8}}(x'')$. Moreover, since $d(x)\le \frac{\sigma}{2}\le (1-2\delta^{\frac{1}{80}})\sigma$,  we know for $x'''=x''+\frac{\xi-x''}{|\xi-x''|}\frac{d(x)}{16}$, there holds
    $$B_{\frac{d(x)}{8}}(x'')\cap B_{(1-2\delta^{\frac{1}{80}})\sigma}(\xi)\cap T\supset B_{\frac{d(x)}{16}}(x''')\cap T.$$

  In the case $x'\in B_{(1-2\delta^{\frac{1}{80}})\sigma}(\xi)\cap T$, by letting $x'''=x'+\frac{\xi-x'}{|\xi-x'|}\frac{d(x)}{16}$ we also get
  $$|x'''-\xi|+\frac{d(x)}{16}=max\{|x'-\xi|, \frac{d(x)}{8}-|x'-\xi|\}\le (1-2\delta^{\frac{1}{80}})\sigma$$
  and
  $$B_{\frac{d(x)}{16}}(x''')\cap T\subset B_{(1-2\delta^{\frac{1}{80}})\sigma}(\xi)\cap B_{\frac{d(x)}{4}}(x')\cap T.$$
   Thus in either case,  by (\ref{largedomain}) we know,
 \begin{align}\label{dsmall}
 \mathcal{H}^2(B_{\frac{d(x)}{2}}(x)\cap graphf \cap B_{\sigma}(\xi))\ge \int_{B_{\frac{d(x)}{16}}(x''')\cap T}\sqrt{1+|\nabla f|^2(y')}dy'\ge \frac{\pi d^2(x)}{2^8}.
 \end{align}
 But by the definition of $d(x)$, we know $B_{\frac{d(x)}{2}}(x)\cap spt\mu_V\cap B_\sigma(\xi)=\emptyset$, so by (\ref{smallmeasure}), we get
 \begin{align}\label{sigmalarge}
 \mathcal{H}^2(B_{\frac{d(x)}{2}}(x)\cap graphf \cap B_{\sigma}(\xi))\le\mathcal{H}^2((graphf\backslash sptV)\cap B_{2\sigma}(\xi))\le 2^{76}\delta^{\frac{1}{32}}\pi (2\sigma)^2.
 \end{align}
 Combining (\ref{dsmall}) and (\ref{sigmalarge}), we know $d(x)\le 2^{43}\delta^{\frac{1}{64}}\sigma\le 2^{43}\delta^{\frac{1}{80}}\sigma$.

 Moreover, since $\delta\le \delta'_6=\frac{1}{2^{3536}k^{80}}\le \frac{1}{2^{3520}}$, we know  $d(x)\le 2^{43}\delta^{\frac{1}{80}}\sigma<\frac{\sigma}{2}$. Thus by the definition of $d(x)$ we know
 $d(x,spt\mu_V\cap B_{\sigma}(\xi))=d(x)\le 2^{41}\delta^{\frac{1}{64}}\sigma$ and hence
 \begin{align}\label{anothersemi}
 d(x',spt\mu_V\cap B_{\sigma}(\xi))\le |f(x')|+d(x,spt\mu_V)\le\delta^{\frac{1}{80}}(2\sigma)+2^{43}\delta^{\frac{1}{80}}\sigma\le 2^{44}\delta^{\frac{1}{80}}\sigma.
 \end{align}
 Combining (\ref{semi}) and (\ref{anothersemi}), we get
 \begin{align*}
 \sigma^{-1}d_{\mathcal{H}}(spt\mu_V\cap B_{\sigma}(\xi), T\cap B_{\sigma}(\xi))\le \min\{2^{15}\delta^{\frac{1}{16}},2^{44}\delta^{\frac{1}{80}}\}=2^{44}\delta^{\frac{1}{80}}.
 \end{align*}

  This is the complete Reifenberg condition (\ref{reifenberg condition}). And the second part of  Theorem \ref{Holder Regularity} is just a restatement Reifenberg's Theorem \ref{reifenberg theorem}.
\end{proof}

\begin{rmk}
We do not know whether some Lipschitz regularity hold under the same condition. See
Corollary \ref{removability} for some positive evidence.
\end{rmk}

\section{The Density Identity and Topological Finiteness} \label{application}
\subsection{The Density Formula}
\

This section is the start point of this paper: we are asking what is the behavior of the inverting of a minimal surface? 
Our first observation is the following density formula which explains the meaning of $\Theta(\Sigma,\infty)$ in the inverting setting. It turns out the result does not depend on the minimal surface equation.

Assume $\Sigma\subset \mathbb{R}^n$ is an immersed surface, we denote the immersion by $f:\Sigma\to \mathbb{R}^n$ and simply call $f:\Sigma\to \mathbb{R}^n$ an immersed surface. We also abuse the notations $\Sigma$ and $f(\Sigma)$ and use $\mathcal{H}^2(B_r(0)\cap \Sigma)$ to mean the Hausdorff measure of the intersection of the extrinsic ball $B_r(0)$ with $f(\Sigma)$. By $d\mu_g$ we mean the volume form of the induced metric $g=f^{*}g_{\mathbb{R}^{n}}$.

\begin{lemma}[\textbf{Density Formula}]\label{Density Formula}
Assume $f:\Sigma\to \mathbb{R}^{n}$ is a properly immersed surface satisfying
\begin{align}\label{finite willmore}
\int_{\Sigma}|H|^2d\mu_g<+\infty
\end{align}
and
\begin{align}\label{lowerdendity}
\Theta_*(\Sigma,\infty)=\liminf_{r\to \infty}\frac{\mathcal{H}^2(B_{r}(0)\cap \Sigma)}{\pi r^2}< +\infty.
\end{align}
Let $h:\Sigma\to \mathbb{R}^n$ be the  inverted  surface, that is,
$h(x)=\frac{f(x)}{|f(x)|^2},\ \forall x\in\Sigma$.
Denote $\tilde{\Sigma}=h(\Sigma)$, $\tilde{H}$=the mean curvature of $\tilde{\Sigma}$ and $\tilde{g}=dh\otimes dh$. Then
we have the locally antisymmetric transformation formula
\begin{align}\label{localinvariant}
\big(\frac{|\tilde{H}|^2}{16}-\big|\frac{\tilde{H}}{4}+\frac{\tilde{\nabla}^{\bot}\tilde{r}}{\tilde{r}}\big|^2\big)d\mu_{\tilde{g}}
=-\big(\frac{|H|^2}{16}-\big|\frac{H}{4}+\frac{\nabla^{\bot}r}{r}\big|^2 \big)d\mu_g
\end{align}
 for $r=|f|$ and $\tilde{r}=|h|$.
Moreover, the density
$$\Theta(\Sigma,\infty):=\lim_{r\to+\infty}\frac{\mathcal{H}^2(\Sigma\cap B_r(0))}{\pi r^2}$$
at infinity is well-defined and satisfies the global representation formula
   \begin{align}\label{density formula}
   \int_{\tilde{\Sigma}\backslash \{0\}}\big(\frac{|\tilde{H}|^2}{16}-\big|\frac{\tilde{H}}{4}+\frac{\tilde{\nabla}^{\bot}\tilde{r}}{\tilde{r}}\big|^2\big)d\mu_{\tilde{g}}=
     \begin{cases}
      \pi\Theta(\Sigma,\infty) & 0\notin \Sigma,\\
      \pi(\Theta(\Sigma,\infty)-\Theta(\Sigma,0)) & 0\in \Sigma,
     \end{cases}
  \end{align}
where $\Theta(\Sigma, 0)=\lim_{r\to 0}\frac{\mu_g(\Sigma\cap B_r(0))}{\pi r^2}$.
\end{lemma}
\begin{proof}
Denote $\tilde{g}=\langle\frac{\partial h}{\partial x^i},\frac{\partial h}{\partial x^j}\rangle dx^i\otimes dx^j$.
Then,
$$ h_i=|f|^{-2}f_i-2f|f|^{-4}\langle f,f_i\rangle,$$
\begin{align*}
h_{ij}=|f|^{-2}f_{ij}&-2|f|^{-4}\langle f,f_j\rangle f_i-2|f|^{-4}\langle f,f_i\rangle f_j+8|f|^{-6}f\langle f,f_i\rangle \langle f,f_j\rangle\\
                   &-2|f|^{-4}f\langle f,f_{ij}\rangle-2|f|^{-4}f\langle f_i,f_j\rangle,
\end{align*}
and
\begin{align*}
\tilde{g}_{ij}
              =\langle h_i,h_j\rangle=|f|^{-4}g_{ij},\ \ \ \tilde{g}=|f|^{-4}g\ \ \ \ \ \text{ and }\ \ \ \  \tilde{g}^{ij}=|f|^4g^{ij}.
\end{align*}
So
\begin{align*}
\tilde{H}=\tilde{g}^{ij}\big(h_{ij}-\langle h_{ij},h_k\rangle \tilde{g}^{kl}h_l\big)
          =|f|^4g^{ij}h_{ij}-|f|^8\langle g^{ij}h_{ij}, h_k\rangle g^{kl}h_l,
\end{align*}
where
\begin{align*}
g^{ij}h_{ij}=&|f|^{-2}g^{ij}f_{ij}-4|f|^{-4}g^{ij}\langle f,f_j\rangle f_i+8|f|^{-6}|f^\top|^2f\\
            &-2|f|^{-4}\langle f, g^{ij}f_{ij}\rangle-4|f|^{-4}f,
\end{align*}
and
\begin{align*}
\langle g^{ij}h_{ij}, h_k\rangle =&\big(|f|^{-2}g^{ij}f_{ij}-4|f|^{-4}g^{ij}\langle f,f_j\rangle f_i+8|f|^{-6}|f^\top|^2f\\
            &-2|f|^{-4}\langle f, g^{ij}f_{ij}\rangle-4|f|^{-4}f\big)\cdot\big(|f|^{-2}f_k-2|f|^4\langle f,f_k\rangle f\big)\\
            =&|f|^{-4}\langle g^{ij}f_{ij},f_k\rangle.
\end{align*}
Thus
\begin{align}
\tilde{H}=&|f|^4(|f|^{-2}g^{ij}f_{ij}-4|f|^{-4}g^{ij}\langle f,f_j\rangle f_i+8|f|^{-6}|f^{\top}|^2f-2|f|^{-4}f\langle f,g^{ij}f_{ij}\rangle\nonumber\\
         &-4|f|^{-4}f)-|f|^4\langle g^{ij}f_{ij},f_k\rangle g^{kl}(|f|^{-2}f_l-2|f|^{-4}\langle f,f_l\rangle f)\nonumber\\
         =&|f|^2H-4f^{\top}+8|f|^{-2}|f^{\top}|^2f-2f\langle f,g^{ij}f_{ij}\rangle-4f+2\langle g^{ij}f_{ij},f^{\top}\rangle f\nonumber\\
         =&|f|^2H-4f^{\top}+8|f|^{-2}|f^{\top}|^2f-2f\langle H,f\rangle-4f\label{generalinverth}
\end{align}
where in the last step we use the equation $\langle g^{ij}f_{ij},f-f^{\top} \rangle=\langle \big(g^{ij}f_{ij}\big)^{\bot},f\rangle=\langle H,f\rangle$.

Since $h=\frac{f}{|f|^2}$, we know $|h|^2=\frac{1}{|f|^2}$, $\tilde{g}^{ij}=|f|^4g^{ij}=\frac{1}{|h|^4}g^{ij}$ and $f=\frac{h}{|h|^2}$, $f_i=|h|^{-2}h_i-2|h|^{-4}\langle h,h_i\rangle h$, $f^{\bot}=f-g^{ij}\langle f,f_i\rangle f_j$. By
$$\langle f,f_i\rangle=\langle \frac{h}{|h|^2}, \frac{1}{|h|^2}h_i-\frac{2}{|h|^4}\langle h, h_i\rangle h\rangle=\frac{\langle h, h_i\rangle}{|h|^4}-\frac{2\langle h,h_i\rangle}{|h|^4}=-\frac{\langle h,h_i\rangle}{|h|^4},$$
we have
\begin{align*}
f^{\bot}
=\frac{h}{|h|^2}+|h|^4\tilde{g}^{ij}\frac{\langle h,h_i\rangle}{|h|^4}(\frac{h_j}{|h|^2}-\frac{2\langle h,h_j\rangle h}{|h|^4})
=\frac{h^{\top}}{|h|^2}+\frac{|h|^2-2|h^{\top}|^2}{|h|^4}h.
\end{align*}
Thus
\begin{align}\label{bot}
|f^{\bot}|^2
=\frac{|h^{\top}|^2}{|h|^4}+\frac{2(|h|^2-2|h^{\top}|^2)}{|h|^6}|h^{\top}|^2+\frac{(|h|^2-2|h^{\top}|^2)^2}{|h|^6}
=\frac{|h^{\bot}|^2}{|h|^4},
\end{align}
$$f^{\top}=f-f^{\bot}=-\frac{h^{\top}}{|h|^2}+\frac{2|h^{\top}|^2h}{|h|^4}\ \ , \ \ |f^{\top}|^2=|f|^2-|f^{\bot}|^2
=\frac{|h^{\top}|^2}{|h|^4},$$
and
\begin{align}\label{tildeh}
\tilde{H}&=|f|^2H-2\langle H,f\rangle f-4f^{\top}+\frac{8|f^{\top}|^2}{|f|^2}f-4f\nonumber\\
         &=\frac{H}{|h|^2}-\frac{2}{|h|^4}\langle H,h\rangle h+4(\frac{|h|^2h^{\top}-2|h^{\top}|^2h}{|h|^4})+\frac{8\frac{|h^{\top}|^2}{|h|^4}}{\frac{|h|^2}{|h|^4}}\frac{h}{|h|^2}-4\frac{h}{|h|^2}\nonumber\\
         &=\frac{H}{|h|^2}-\frac{2}{|h|^4}\langle H,h\rangle h-4\frac{h^{\bot}}{|h|^2}.
\end{align}
Moreover, for $\tilde{r}=|h|$, we know $\tilde{\nabla}^{\bot}\tilde{r}=\frac{h^{\bot}}{|h|}$. So,  by (\ref{bot}) and (\ref{tildeh})  we get,
\begin{align}\label{bracketterm}
\langle \tilde{H}, \frac{h^{\bot}}{|h|^2}\rangle=\langle \frac{H}{|h|^2}-\frac{2}{|h|^4}\langle H,h\rangle h-4\frac{h^{\bot}}{|h|^2},\frac{h}{|h|^2} \rangle=-|f|^2\langle H,f\rangle-4|f^{\bot}|^2
\end{align}
and
\begin{align*}
-\big(\frac{|\tilde{H}|^2}{16}-\big|\frac{\tilde{H}}{4}+\frac{\tilde{\nabla}^{\bot}\tilde{r}}{\tilde{r}}\big|^2\big)d\mu_{\tilde{g}}
&=\big(\frac{1}{2}\langle\tilde{H},\frac{h^{\bot}}{|h|^2}\rangle+\big|\frac{h^{\bot}}{|h|^2}\big|^2\big)|f|^{-4}d\mu_g\\
&=\big(-\frac{1}{2}|f|^2\langle H,f\rangle-2|f^{\bot}|^2+|f^{\bot}|^2\big)|f|^{-4}d\mu_g\\
&=-\big(\frac{|f^{\bot}|^2}{|f|^4}+\frac{1}{2}\langle H,\frac{f^{\bot}}{|f|^2}\rangle\big)d\mu_g\\
&=\big(\frac{|H|^2}{16}-\big|\frac{H}{4}+\frac{\nabla^{\bot}r}{r}\big|^2 \big)d\mu_g.
\end{align*}

The following argument belongs to \cite[Appendix]{KS}.  By (\ref{lowerdendity}), (\ref{finite willmore}) and  Corollary \ref{integral discription of upper density at infinity}, we know $\Theta^*(\Sigma,\infty)<+\infty$ and
\begin{align}\label{integral finite}
\int_{\Sigma}\big|\frac{\nabla^{\bot}r}{r}\big|^2<+\infty.
\end{align}
Thus for any $\varepsilon>0$, there exists $\rho_0>0$ such that for any $\rho\ge \rho_0$, we have
\begin{align*}
\int_{\Sigma\backslash B_{\rho_0}}|H|^2d\mu_g\le \varepsilon \ \ \ \ \text{ and } \ \ \ \ \frac{\mathcal{H}^2(\Sigma\cap B_\rho(0))}{\pi \rho^2}\le \Theta^*(\Sigma,\infty)+\varepsilon.
\end{align*}
On the one hand,
\begin{align*}
|\frac{1}{2\rho^2}\int_{B_{\rho}(0)}\langle r\nabla^{\bot} r,H \rangle d\mu_g|
&\le \frac{1}{2\rho}\int_{B_{\rho_0}(0)}|H|+\frac{\pi\varepsilon^{\frac{1}{2}}}{2}(\Theta^*(\Sigma,\infty)+\varepsilon)^{\frac{1}{2}}.
\end{align*}
Letting $\rho\to \infty$ first and then $\varepsilon\to 0$, we get
\begin{align}\label{rhovanishing}
\lim_{\rho\to \infty}\frac{1}{2\rho^2}\int_{B_{\rho}(0)}\langle r\nabla^{\bot} r,H \rangle d\mu_g=0.
\end{align}
On the other hand, by
\begin{align*}
|\frac{1}{2\sigma^2}\int_{B_{\sigma}(0)}\langle r\nabla^{\bot} r,H \rangle d\mu_g|
\le \frac{1}{2}\big(\frac{\mathcal{H}^2(\Sigma\cap B_\sigma(0))}{\sigma^2}\int_{\Sigma\cap B_\sigma(0)}|H|^2\big)^{\frac{1}{2}},
\end{align*}
we also know
\begin{align}\label{sigmavanishing}
\lim_{\sigma\to 0}\frac{1}{2\sigma^2}\int_{B_{\sigma}(0)}\langle r\nabla^{\bot} r,H \rangle d\mu_g=0
\end{align}
So, by (\ref{finite willmore})(\ref{integral finite})(\ref{rhovanishing})(\ref{sigmavanishing}) and letting $\rho\to \infty$ and $\sigma\to 0$ in the monotonicity formula (\ref{monotonicity equality}), we know $\Theta(\Sigma,\infty)$ is well-defined and satisfies
\begin{align}\label{density diff}
\pi(\Theta(\Sigma,\infty)-\Theta(\Sigma,0))
&=-\int_{\Sigma}\big(\frac{|H|^2}{16}-\big|\frac{H}{4}+\frac{\nabla^{\bot}r}{r}\big|^2\big)d\mu_g\\
&=\int_{\tilde{\Sigma}\backslash\{0\}}\big(\frac{|\tilde{H}|^2}{16}-\big|\frac{\tilde{H}}{4}+\frac{\tilde{\nabla}^{\bot}\tilde{r}}{\tilde{r}}\big|^2\big)d\mu_{\tilde{g}},\nonumber
\end{align}
where in the last line we use (\ref{localinvariant}).
\end{proof}

\begin{rmk}
In the special case of minimal surfaces,  the density formula goes like
\begin{align*}
   \int_{\tilde{\Sigma}\backslash \{0\}}|\tilde{H}|^2d\mu_{\tilde{g}}
   =16\int_{\Sigma}\big|\frac{\nabla^{\bot}r}{r}\big|^2d\mu_g=
     \begin{cases}
      16\pi\Theta(\Sigma,\infty) & 0\notin \Sigma,\\
      16\pi(\Theta(\Sigma,\infty)-\Theta(\Sigma,0)) & 0\in \Sigma.
     \end{cases}
\end{align*}
It means the density of a minimal surface can dominate the Willmore energy of its inverted surface $\tilde{\Sigma}$. But in general, the inverted surface $\tilde{\Sigma}$ has singularity at the inverted point $0$ and the density formula can not dominate the topology of geometry(say total curvature) of $\tilde{\Sigma}$. For example, the family of Scherk's singly-periodic minimal surfaces have density two at infinity, but they all have infinity genuses.
\end{rmk}
\begin{rmk} As it is seen, the locally antisymmetric transformation formula (\ref{localinvariant}) and then density formula follows easily from direct calculation.  But how such a term occurs? Here we give an explanation in the setting of conformal deformation of submanifolds.
Recall there are two conformal invariances for surfaces, the extrinsic local one
\begin{align}\label{tracefree}
|A-\frac{H}{n}g|^2_gd\mu_g=|\tilde{A}-\frac{\tilde{H}}{n}\tilde{g}|^2_{\tilde{g}}d\mu_{\tilde{g}}
\end{align}
and the intrinsic global one---the Gauss-Bonnet formula
$$\int_\Sigma Kd\mu_g=\int_{\Sigma}\tilde{K}d\mu_{\tilde{g}}.$$
Under conformal setting, the global Gauss-Bonnet formula has a local explanation. Assume $\tilde{g}=e^{2u}g$ is a conformal metric on a closed Riemann surface $(\Sigma,g)$.  Applying Stokes' formula to the Yamabe equation
$$\triangle_g u-K+\tilde{K}e^{2u}=0,$$
we get
$$\int_\Sigma Kd\mu_g=\int_\Sigma\tilde{K}e^{2u}d\mu_g=\int_{\Sigma}\tilde{K}d\mu_{\tilde{g}}.$$

For the same reason, in higher dimensional, assume $\tilde{g}=u^{\frac{4}{n-2}}g$ and apply the Stokes formula to the Yamabe equation
$$\triangle_g u-\frac{n-2}{4(n-1)}Su+\frac{n-2}{4(n-1)}\tilde{S}u^{\frac{n+2}{n-2}}=0.$$
We get
 \begin{align}\label{conformal}
 \int_{M}Sud\mu_g=\int_{M}\tilde{S}\tilde{u}d\mu_{\tilde{g}},
 \end{align}
 where $\tilde{u}=u^{-1}$ satisfies $g=\tilde{u}^{\frac{4}{n-2}}\tilde{g}$. Note both sides contain the conformal factors $(u,\tilde{u})$. So, in high dimension, the invariance is not in a conformal class, but just for a conformal pair $(g,\tilde{g})$.  With this experience, we guess the corresponding extrinsic invariant should also admit the shape of
 \begin{align}\label{star}
   \star ud\mu_g=\tilde{\star}\tilde{u}d\mu_{\tilde{g}}\text{ (local) }，
 \end{align}
 or
  \begin{align}\label{globalstar}
   \int_{M}\star ud\mu_g=\int_M\tilde{\star}\tilde{u}d\mu_{\tilde{g}} \text{ (global) }.
 \end{align}
 For example, for high dimensional analogue of (\ref{tracefree}), we assume $M^n\subset N^{n+k}$ and  $(G,\tilde{G})$ are a pair of conformal metrics on $N^{n+k}$ with conformal factors $(U,\tilde{U})$, i.e., $\tilde{G}=U^{\frac{4}{n-2}}G$(note the index $n=dim M$)  and $\tilde{U}=U^{-1}$.  Denote $u=U|_M, \tilde{u}=\tilde{U}|_M$ and assume $(g,\tilde{g})$ are the induced metrics of $M\subset \big(N,(G,\tilde{G})\big)$. Then $\tilde{g}=u^{\frac{4}{n-2}}g$ and direct calculus shows high dimensional analogue of (\ref{tracefree}) is of type (\ref{star}):
\begin{align*}
|A-\frac{H}{n}g|^2_gud\mu_g=|\tilde{A}-\frac{\tilde{H}}{n}\tilde{g}|^2_{\tilde{g}}\tilde{u}d\mu_{\tilde{g}}.
\end{align*}
This is a local one. A natural question is, what extrinsic global invariance is corresponding to the intrinsic global invariance (\ref{conformal}). We take $n\ge 3$ as an example.  For this, we take the  trace of the restriction of Ricci tensor of $G$ on $M$, i.e., denote
$$S^G_g=tr_gRic(N,G)$$
and call it the extrinsic scalar curvature. The goal is to find the invariance of type (\ref{globalstar}) involving $R_g^G$.  As in the intrinsic case, the first step is to calculate the equation of the extrinsic scalar curvature when the background metric deforms conformally. The result is
 \begin{align}\label{exyamabe}
div^M\nabla U+\frac{n}{n-2}\frac{|\nabla^{\bot}U|^2}{u}-\frac{n-2}{4(n-1)}S^G_gu+\frac{n-2}{4(n-1)}S^{\tilde{G}}_{\tilde{g}}u^{\frac{n+2}{n-2}}=0,
\end{align}
where，$div^M\nabla U$ means the extrinsic divergence of the restriction of the gradient of $U$ on $M$ and  $\nabla^{\bot}U$ represent the projection of $\nabla U$ to the normal bundle $T^{\bot}M$.  Since (\ref{exyamabe}) reduces to the Yamabe equation when $M=N$, we call it extrinsic Yamabe equation. The next is to apply the Stokes formula to the extrinsic Yamabe equation. Note the extrinsic divergence theorem goes like
$$\int_{M}div^M\nabla U d\mu_g=-\int_{M}\nabla U\cdot Hd\mu_g,$$
where $H$ is the mean curvature of the submanifold $(M,g)\subset (N,G)$. We get the global equation
\begin{align}\label{nonsymetry}
\tilde{C}:=\int_{M}S^{\tilde{G}}_{\tilde{g}}\tilde{u}d\mu_{\tilde{g}}=&\int_{M}S^G_gud\mu_g\nonumber\\
&+\int_{M}\big(\frac{4(n-1)}{n-2}\langle \frac{\nabla U}{u},H\rangle -\frac{4(n-1)n}{(n-2)^2}\frac{|\nabla^{\bot}U|^2}{u^2}\big)ud\mu_g\nonumber\\
=&:C+Q'.
\end{align}
This equation looks not so symmetrically as we expected.  To make (\ref{nonsymetry}) to possess the symmetry of type (\ref{globalstar}), we guess the term $Q'$ is a global antisymmetric term, i.e, $\tilde{Q}'=-Q'$. If so, then (\ref{nonsymetry}) become the symmetric form
$$\tilde{C}+\frac{\tilde{Q}'}{2}=C+\frac{Q'}{2}.$$
It turns out that $Q'$ is not only globally antisymmetric, but also comes form a local conformal antisymmetry:
\begin{align}\label{localantisym}
\tilde{Q}:&=\big(\frac{1}{n-2}\langle \frac{\tilde{\nabla}\tilde{U}}{\tilde{u}},\tilde{H}\rangle_{\tilde{g}}-\frac{n}{(n-2)^2}|\frac{\tilde{\nabla}^{\bot}\tilde{U}}{\tilde{u}}|_{\tilde{g}}^2\big)\tilde{u}d\mu_{\tilde{g}}\nonumber\\
&=-\big(\frac{1}{n-2}\langle \frac{\nabla U}{u},H\rangle_{g}-\frac{n}{(n-2)^2}|\frac{\nabla^{\bot}U}{u}|_{g}^2\big)ud\mu_{g}=-Q.
\end{align}
So (\ref{nonsymetry}) becomes the symmetric form of type (\ref{globalstar}), i.e.,
\begin{align}
\int_M(S^{\tilde{G}}_{\tilde{g}}+T^{\tilde{G}}_{\tilde{g}})\tilde{u}d\mu_{\tilde{g}}=\int_M(S^G_g+T^G_g)ud\mu_g,
\end{align}
where, $T^G_g=2(n-1)Q=\frac{2(n-1)}{n-2}\langle \frac{\nabla U}{u},H\rangle_{g}-\frac{2n(n-1)}{(n-2)^2}|\frac{\nabla^{\bot}U}{u}|_{g}^2$.

The above calculation is in a compact manifold, but the antisymmetry (\ref{localantisym}) is a local form, which also holds in noncompact ambient space. Especially, when we are caring about submanifolds in $\mathbb{R}^{n+k}$ and the conformal factor is induced by the inversion, (\ref{localantisym}) coincides with the locally antisymmetric transformation formula 
(\ref{localinvariant}) in dimension $n=2$, which is a key observation in getting the density identity.
\end{rmk}

\begin{rmk}\label{infinitedensity}
In the case (\ref{lowerdendity}) does not holds, i.e., $\Theta_*(\Sigma,\infty)=+\infty$, its natural to define $\Theta(\Sigma,\infty)=+\infty$. So, by the lemma, for a properly immersed surface in $\mathbb{R}^n$ with$\int_{\Sigma}|H|^2d\mu_g<+\infty$, the density
$
\Theta(\Sigma,\infty)=\lim_{r\to\infty}\frac{\mathcal{H}^2(\Sigma\cap B_r(0))}{\pi r^2}
$
is always well-defined, whether it is finite of infinite. In this sense, Lemma \ref{Density Formula} holds without the assumption of (\ref{lowerdendity}). Only in the case $\Theta(\Sigma,\infty)=+\infty$, by (\ref{localinvariant}) and Corollary \ref{integral discription of upper density at infinity}, both side of (\ref{density formula}) are infinite.
\end{rmk}

\subsection{The Density Identity}
\

Firstly, we need the following weak(in varifold sense) removability of singularity.

\begin{prop}\label{weak removable}
Assume $f:\Sigma\to \mathbb{R}^{n}$ is a properly immersed surface satisfying (\ref{finite willmore}) and (\ref{lowerdendity}) and $\tilde{\Sigma}=h(\Sigma)$ is its inverted surface. Then for any $r\in (0,\infty)$,
\begin{align*}
\mu_{\tilde{g}}(B_r(0)\backslash\{0\})\le \frac {C e^{4}}{(e^2-1)}\pi r^2,
\end{align*}
where $C=9\Theta_*(\Sigma,\infty)+\frac{59}{16\pi}\int_{\Sigma}|H|^2d\mu_g$. And we have
\begin{align*}
\int_{\tilde{\Sigma}\backslash \{0\}}|\tilde{H}|^2d\mu_{\tilde{g}}<+\infty.
\end{align*}
Moreover,  if we extend $\mu_{\tilde{g}}$ and $\tilde{H}$ trivially across $0\in \mathbb{R}^n$, then for vector field $X\in C_0^1(\mathbb{R}^n,\mathbb{R}^n)$ (do not need to be supported in $\mathbb{R}^n\backslash\{0\}$), we have
\begin{align*}
\int_{\mathbb{R}^n}div^{\tilde{\Sigma}}Xd\mu_{\tilde{g}}=-\int_{\mathbb{R}^n}\langle X, \tilde{H}\rangle d\mu_{\tilde{g}}.
\end{align*}
That is, $\tilde{\Sigma}$ is a varifold in $\mathbb{R}^n$ with generalized mean curvature $\tilde{H}\in L^{2}(\mu_{\tilde{g}})$.
\end{prop}
\begin{proof}
By (\ref{finite willmore}), (\ref{lowerdendity}) and (\ref{everyradius}) in Corollary \ref{integral discription of upper density at infinity}, we know for any $\rho\in (0,\infty)$,
$$\frac{\mathcal{H}^2(B_\rho(0)\cap \Sigma)}{\pi \rho^2}\le C,$$
where $C=9\Theta_*(\Sigma,\infty)+\frac{59}{16\pi}\int_{\Sigma}|H|^2d\mu_g$.
Since $\tilde{g}=\frac{1}{|f|^4}g$, we know $d\mu_{\tilde{g}}=\frac{1}{|f|^4}d\mu_g$. So,  for $r=e^{-t}>0$,
\begin{align}\label{localuperdensity}
\mu_{\tilde{g}}(\tilde{\Sigma}\cap B_r\backslash \{0\})
&=\lim_{\varepsilon\to 0}\int_{\tilde{\Sigma}\cap (B_r\backslash B_{\varepsilon})}d\mu_{\tilde{g}}=\lim_{\varepsilon\to 0}\int_{\Sigma\cap (B_{\frac{1}{\varepsilon}}\backslash B_{\frac{1}{r}})}\frac{1}{|f|^4}d\mu_{g}\nonumber\\
&=\sum_{k=1}^{\infty}\int_{\Sigma\cap(B_{e^{t+k}}\backslash B_{e^{t+(k-1)}})}\frac{1}{r^4}d\mu_g\nonumber\\
&\le \sum_{k=1}^{\infty}\frac{C\pi e^{2(t+k)}}{e^{4(t+(k-1))}}=\frac{C\pi e^{4}}{(e^2-1)}r^2.
\end{align}
 By (\ref{generalinverth}), we note
\begin{align}\label{observation}
\tilde{H}&=|f|^2H-2\langle H,f\rangle f-4f^{\top}+\frac{8|f^{\top}|^2}{|f|^2}f-4f\nonumber\\
&=|f|^2H-2\langle H,f\rangle f+4f^{\bot}-\frac{8|f^{\bot}|^2}{|f|^2}f.
\end{align}
Thus
\begin{align*}
|\tilde{H}|^2
&\le 320(|f|^4|H|^2+|f^{\bot}|^2)
\end{align*}
and
\begin{align*}
\int_{\tilde{\Sigma}\cap(B_\rho(0)\backslash \{0\})}|\tilde{H}|^2d\mu_{\tilde{g}}
&\le 320\int_{\Sigma\backslash B_{\frac{1}{\rho}}}(|f|^4|H|^2+|f^{\bot}|^2)|f|^{-4}d\mu_g\\
&\le 320\int_{\Sigma}\big(|H|^2+\big|\frac{\nabla^{\bot}r}{r}\big|^2\big)d\mu_g.
\end{align*}
By Corollary \ref{integral discription of upper density at infinity} again, the right hand term is finite. So, letting $\rho\to+\infty$ and we get
\begin{align}\label{finiteinveresewillmore}
\int_{\tilde{\Sigma}\backslash \{0\}}|\tilde{H}|^2d\mu_{\tilde{g}}\le  320\int_{\Sigma}\big(|H|^2+\big|\frac{\nabla^{\bot}r}{r}\big|^2\big)<+\infty.
\end{align}
Finally, by (\ref{localuperdensity}), (\ref{finiteinveresewillmore}) and cut-off argument(see \cite[Appendix]{KS}), we know $\tilde{\Sigma}$ is a varifold in $\mathbb{R}^n$ with generalized mean curvature $\tilde{H}\in L^{2}(\mu_{\tilde{g}})$.
\end{proof}
With this proposition, we know the monotonicity formula holds for the varifold $\tilde{\Sigma}$ with generalized mean curvature $\tilde{H}$ and can be used to show the following density identity.
\begin{lemma}[Density identity]\label{density identity}
Assume $f:\Sigma\to \mathbb{R}^{n}$ is a properly immersed surface satisfying (\ref{finite willmore}) and (\ref{lowerdendity}) and $\tilde{\Sigma}=h(\Sigma)$ is its inverted surface. If the base point $0\notin \Sigma$,  then
\begin{align*}
\Theta(\tilde{\Sigma},0):=\lim_{\sigma\to 0}\frac{\mu_{\tilde{g}}(\tilde{\Sigma}\cap B_{\sigma}(0))}{\sigma^2}=\Theta(\Sigma,\infty)\ge 1.
\end{align*}
\end{lemma}
\begin{proof}
In this case, for $0<\sigma<\rho<\infty$, we have the monotonicity formula
\begin{align}\label{monoton}
\frac{\mu_{\tilde{g}}(\tilde{\Sigma}\cap B_{\sigma})}{\sigma^2}
&=\frac{\mu_{\tilde{g}}(\tilde{\Sigma}\cap B_{\rho})}{\rho^2}+\frac{1}{2\rho^2}\int_{\tilde{\Sigma}\cap B_{\rho}}\langle \tilde{r}\tilde{\nabla}^{\bot}\tilde{r}, \tilde{H}\rangle-\frac{1}{2\sigma^2}\int_{\tilde{\Sigma}\cap B_{\sigma}}\langle \tilde{r}\tilde{\nabla}^{\bot}\tilde{r}, \tilde{H}\rangle\nonumber\\
     &+\frac{1}{16}\int_{\tilde{\Sigma}\cap (B_{\rho}\backslash B_{\sigma})}|\tilde{H}|^2-\int_{\tilde{\Sigma}\cap (B_{\rho}\backslash B_{\sigma})}|\frac{\tilde{\nabla}^{\bot} \tilde{r}}{\tilde{r}}+\frac{\tilde{H}}{4}|^2
\end{align}
On the one hand, by (\ref{finiteinveresewillmore}) and (\ref{localuperdensity}),  we know
$$\lim_{\sigma\to 0}W(\sigma):=\lim_{\sigma\to0}\int_{\tilde{\Sigma}\cap B_{\sigma}}|\tilde{H}|^2= 0$$
and
$$\lim_{\sigma\to 0}|\frac{1}{2\sigma^2}\int_{\tilde{\Sigma}\cap B_{\sigma}}\langle \tilde{r}\tilde{\nabla}^{\bot}\tilde{r}, \tilde{H}\rangle|\le \lim_{\sigma\to 0}\frac{1}{2}(\frac{\mu_{\tilde{g}}(\tilde{\Sigma}\cap B_{\sigma})}{\sigma^2})^{1/2}W(\sigma)^{1/2}= 0.$$
On the other hand, the properness of $f$ and $0\notin \Sigma$ implies $\tilde{\Sigma}\backslash B_\sigma(0)$ is compact. So we have
$$\lim_{\rho\to \infty}\frac{\mu_{\tilde{g}}(\tilde{\Sigma}\cap B_{\rho})}{\rho^2}= 0$$
and
$$\lim_{\rho\to +\infty}|\frac{1}{2\rho^2}\int_{\tilde{\Sigma}\cap B_{\rho}}\langle \tilde{r}\tilde{\nabla}^{\bot}\tilde{r}, \tilde{H}\rangle|\le \lim_{\rho\to+\infty}\frac{1}{2}(\frac{\mu_{\tilde{g}}(\tilde{\Sigma}\cap B_{\rho})}{\rho^2})^{1/2}(\int_{\tilde{\Sigma}}|\tilde{H}|^2)^{1/2}= 0.$$
Letting $\rho\to \infty$ and $\sigma\to 0$ in (\ref{monoton}) and applying the density formula (\ref{density formula}), we get
$$\Theta(\tilde{\Sigma},0)=\int_{\tilde{\Sigma}\backslash \{0\}}\big(\frac{|\tilde{H}|^2}{16}-\big|\frac{\tilde{H}}{4}
+\frac{\tilde{\nabla}^{\bot}\tilde{r}}{\tilde{r}}\big|^2\big)d\mu_{\tilde{g}}
=\Theta(\Sigma,\infty).$$

Noting the inverted surface $\tilde{\Sigma}$ is smooth away from 0, by
Lemma \ref{density at each point} and the properness of $f$, we know
\begin{align*}
\Theta(\Sigma,\infty)=\Theta(\tilde{\Sigma},0)\ge \limsup_{y\to 0}\Theta(\tilde{\Sigma},y)\ge 1.
\end{align*}
\end{proof}


\subsection{Topological Finiteness}
\

The density identity and density formula implies the single term $\Theta(\Sigma, \infty)$ can control both the Willmore energy and the local density of the inverted surface. So, when combining with the Allard-Reifenberg type $C^{\alpha}$ regularity Theorem \ref{Holder Regularity}, we can prove the main theorem.
\begin{prop}\label{topolofical rigidity of minimal ends}
Assume $f:\Sigma\to \mathbb{R}^n$ is a properly immersed surface with finite Willmore energy. For any $R>0$, let $\Sigma_1$ be a noncompact connected component of $\Sigma\backslash f^{-1}(B_R(0))$. Then $$\Theta(\Sigma_1,\infty):=\lim_{r\to \infty}\frac{\mathcal{H}^2(\Sigma_1\cap B_r(0)}{\pi r^2}\ge1.$$
 Moreover, there exists an $\varepsilon=\varepsilon(n)>0$ such that if $$\Theta(\Sigma_1,\infty)\le 1+\varepsilon(n),$$
  then there is an $R_2\ge R$ such that for any $r\ge R_2$, $\Sigma_1\backslash f^{-1}(B_r)$ is homeomorphic to $S^1\times \mathbb{R}$ and $f:\Sigma_1\backslash f^{-1}(B_r)\to \mathbb{R}^n$ is embedding.
\end{prop}
\begin{proof}
Since $\Sigma$ is proper, we know $\Sigma_1$ has compact boundary, thus can be extended to be a complete surface in $\mathbb{R}^n$ without boundary by gluing a compact surface $\Sigma_2$ with $\partial \Sigma_2=-\partial\Sigma_1$. So we can assume $\Sigma_1$ to be a surface properly immersed in $\mathbb{R}^n$ without boundary and satisfies  (\ref{finite willmore}).
 Thus by Remark \ref{infinitedensity} and Lemma \ref{density identity}, we know $\Theta(\Sigma_1,\infty):=\lim_{r\to \infty}\frac{\mathcal{H}^2(\Sigma_1\cap B_r(0)}{\pi r^2}$ is well-defined and  $\Theta(\Sigma_1,\infty)\ge1.$

 Moreover, in the case $\Theta(\Sigma_1,\infty)\le 1+\varepsilon$, choose a base point $x_0\notin \Sigma_1$, define $h(x)=\frac{f(x)-x_0}{|f(x)-x_0|^2}+x_0$ and denote the inverted surface by $\tilde{\Sigma}_1=h(\Sigma_1)$. Then by Proposition \ref{weak removable} and Lemma \ref{density identity}, we know $\tilde{\Sigma}_1$ is a rectifiable $2$-varifold in $\mathbb{R}^n$ with generalized mean curvature $\tilde{H}\in L^{2}(\mathbb{R}^n,d\mu_{\tilde{g}})$ and
\begin{align*}
\Theta(\tilde{\Sigma}_1,x_0)
=\frac{1}{16}\int_{\tilde{\Sigma}_1}|\tilde{H}|^2-\int_{\tilde{\Sigma}_1}|\frac{\tilde{\nabla}^{\bot}\tilde{r}}{\tilde{r}}+\frac{\tilde{H}}{4}|^2
=\Theta(\Sigma_1,\infty)
\in [1,1+\varepsilon).
\end{align*}
So, there exists $\rho_0>0$ such that for any $\rho<\rho_0$, we have
\begin{align*}
\frac{\mathcal{H}^2(\tilde{\Sigma}_1\cap B_\rho(x_0))}{\pi \rho^2}\le 1+2\varepsilon,\text{ and  } \int_{\tilde{\Sigma}_1\cap B_{\rho}(x_0)}|\tilde{H}|^2<\varepsilon.
\end{align*}
Since $\tilde{\Sigma}_1$ is smooth outside the base point $x_0$, we know $\Theta(x)\ge 1$ for every $x\in \tilde{\Sigma}_1$.
Taking $\varepsilon=\varepsilon(n)$ small enough and applying Theorem \ref{Holder Regularity}  we know $\tilde{\Sigma}_1\cap B_\sigma(x_0)$ is a topological disk for $\sigma\le \frac{1}{2^{19}}(2\varepsilon)^{\frac{1}{2}}\rho_0$, which implies the conclusion.
\end{proof}

\begin{rmk}\label{end density}
By a geometric measure theory argument of  E.Kuwert, Y.X.Li and R.Sch\"{a}tzle
(see \cite[Appendix]{KS} and \cite{KLS}), it is directly shown $\Theta(\Sigma_1,\infty)\ge 1$ and if $\Theta(\Sigma_1,\infty)<2$, then $\Theta(\Sigma_1,\infty)=1$. We sketch the proof for reader's convenience.
\end{rmk}
\begin{proof}
 We also assume $\Theta(\Sigma_1,\infty)<+\infty$. Extend $\Sigma_1$ to be smooth and boundary free and still denote it by $\Sigma_1$. Take the current $T_r=\big(\frac{1}{r}\big)_{\sharp}\Sigma_1$ and the varifold $\mu_r=\big(\frac{1}{r}\big)_{\sharp}(\mathcal{H}^2\llcorner \Sigma_1)$. By (\ref{rhovanishing}), for any $R>0$,
 $$\lim_{r\to \infty}\|\delta \mu_r\|(B_R(0))=\lim_{r\to \infty}\frac{\int_{\Sigma_1\cap B_{rR}(0)}|H|d\mu_g}{r}=0.$$
 So, by the compactness of varifold \cite{FF60}\cite[Theorem 32.2 and Lemma 26.14]{LS83} and the compactness of integral varifolds\cite{A72},\cite[Theorem 42.7 and Remark 42.8]{LS83}, there exist  an integral current $T_\infty$, a stationary integral  varifold $\mu_\infty$ and a sequence of $r_i\to +\infty$ such that
     $$T_{r_i}\to T_\infty( \text{ weak convergence as currents }),$$
     $$\mu_{r_i}\to \mu_\infty(\text{ weak convergence as varifolds}).$$
  Since $\mu_\infty$ is a Radon measure, we know  for fixed $x$ and $\mathcal{L}^1$-almost every $\rho>0$, $\mu_{\infty}(\partial B_\rho(x))=0$ and
  $$\frac{\mu_\infty(B_{\rho}(x))}{\pi \rho^2}= \lim_{i\to +\infty}\frac{\mathcal{H}^2(\Sigma_1\cap B_{r_i\rho}(x))}{\pi(r_i\rho)^2}=\Theta(\Sigma_1,\infty).$$
 Since $\mu_\infty$ is stationary and integral,  by the monotonicity formula and the upper semi-continuity, we know
 $$\Theta(\Sigma_1,\infty)=\Theta(\mu_\infty,\infty)=\Theta(\mu_\infty,x)\ge\limsup_{y\to x}\Theta(\mu_\infty,y)\ge 1.$$

 Moreover, when $\Theta(\Sigma_1,\infty)<2$, noting $\lim_{i\to \infty}\|\delta \mu_{r_i}\|(B_R(0))=0$ for any $R>0$ and  $\Theta(\mu_{r_i},\infty)\equiv \Theta(\Sigma_1,\infty)<2$, by the same argument as in \cite[Proposition 2.2]{KLS}, we get
 $$\mu_\infty=\mu_{T_\infty}.$$
 So  $\Theta(\mu_{T_\infty},\infty)=\Theta(\mu_\infty,\infty)\in [1,2)$ and by \cite[Theorem 2.1]{KLS}, we know $T_\infty$ is a plane. Thus
 $$\Theta(\Sigma_1,\infty)=\Theta(\mu_\infty,0)=\Theta(\mu_{T_\infty},0)=1.$$
\end{proof}

As a corollary, our main theorem is a global version of the above topological rigidity Proposition \ref{topolofical rigidity of minimal ends}.
\begin{thm}[$\mathbf{Finite\  Topology}$]\label{finite topology}
Assume $f:\Sigma\to \mathbb{R}^n$ is a properly immersed surface with finite Willmore energy.
Then $$e(\Sigma,\infty)\le \Theta(\Sigma,\infty).$$
 Moreover, if we assume
  $$e(\Sigma,\infty)> \Theta(\Sigma, \infty)-1<\infty,$$
  then
\begin{enumerate}[1)]
\item $\Sigma$ has finite topology;
\item $\Theta(\Sigma,\infty)=e(\Sigma,\infty)=:e$ is an integer and $\Sigma$ has exact $e$ ends with density one.
\item $\Sigma$ has finite total curvature, i.e., $\int_{\Sigma}|A|^2d\mu_g<+\infty$;
\item $\Sigma$ is conformal to a closed Riemann surface with $e(\Sigma,\infty)$ points removed.

\end{enumerate}
\end{thm}
\begin{proof}
There is nothing to prove if $\Theta(\Sigma,\infty)=+\infty$. So we assume $\Theta(\Sigma,\infty)<+\infty$.  By the properness, for each $r>0$,  $\Sigma\cap B_{r}(0)$ has the  connected components decomposition  $\Sigma\cap B_{r}(0)=K_r\sqcup \sqcup_{i\in I(r)} \Sigma_{i,r}$, where $K_r$ is the compact part and each $\Sigma_{i,r}$ is noncompact. By Proposition \ref{topolofical rigidity of minimal ends}, we get for each $i\in I(r)$,
$\Theta(\Sigma_{i,r},\infty)\ge 1$. Since these $\{\Sigma_{i,r}\}_{i\in I(r)}$ are disjoint, we know
\begin{align}\label{Ir}
|I(r)|\le \sum_{i\in I(r)}\Theta(\Sigma_{i,r},\infty)\le \Theta(\Sigma,\infty)<+\infty.
\end{align}
Letting $r\to +\infty$, we know $$e(\Sigma,\infty)=\lim_{r\to \infty}|I(r)|\le \Theta(\Sigma,\infty).$$  Moreover, if $e(\Sigma,\infty)> \Theta(\Sigma, \infty)-1$, then there exists $r_0>0$ such that $e(\Sigma,\infty)=|I(r_0)|>\Theta(\Sigma, \infty)-1$. So by (\ref{Ir}) and $\Theta(\Sigma_{i,r_0},\infty)\ge 1$, we know
$$\Theta(\Sigma_{i,r_0},\infty)<2, \forall i\in I(r_0).$$
By Remark \ref{end density} we know in fact $$\Theta(\Sigma_{i,r_0},\infty)=1.$$
Thus $e(\Sigma,\infty)=\Theta(\Sigma,\infty)$ and  by Proposition \ref{topolofical rigidity of minimal ends} again, there exists $r_1>r_0$ such that for every $r\ge r_1$,  each $\Sigma_{i,r_0}\backslash B_r(0)$ is an embedded annulus in $\mathbb{R}^n$.  Take $r$ large enough such that $K_{r_0}\subset B_r(0)$. Then $\Sigma\backslash B_r(0)=\sqcup_{i\in I(r_0)}\big(\Sigma_{i,r_0}\backslash B_r(0)\big)$ consists of $|I(r_0)|=e(\Sigma,\infty)$ many properly embedded annulus. By properness, $\Sigma\cap B_r(0)$ is compact, so $\Sigma$ is homeomorphic to a closed surface with  $e(\Sigma,\infty)$ points removed.
Now, by Ilmanen's local Gauss-Bonnet estimate \cite[Theorem 3]{I95}, we know for each $r<s<\infty$ and $\varepsilon>0$,
\begin{align*}
(1-\varepsilon)\int_{\Sigma\cap B_r(0)}|A|^2d\mu_g
\le \int_{\Sigma\cap B_s(0)}|H|^2d\mu_g+8\pi g(\Sigma\cap B_s(0))+\frac{24\pi D's^2}{\varepsilon(s-r)^2},
\end{align*}
where $g(\Sigma\cap B_s(0))$ is the genus of the closed surface obtained by capping off the boundary of  $\Sigma\cap B_s(0)$ by disks and $D'=\sup_{t\in [r,s]}\frac{\mathcal{H}^2(\Sigma\cap B_t(0))}{\pi t^2}$. Since we have shown $\Sigma$ has finite topology, by letting $s\to \infty$ and then $r\to \infty$ and taking $\varepsilon=\frac{1}{2}$, we get
\begin{align*}
\int_{\Sigma}|A|^2d\mu_g\le 2\int_{\Sigma}|H|^2d\mu_g+16\pi g(\Sigma)+96\pi \Theta(\Sigma,\infty)<+\infty.
\end{align*}
So, by Huber's classification\cite{H57} of complex structures for complete surfaces with finite total curvature, each end of $\Sigma$ is parabolic, i.e., $\Sigma$ is conformal to a closed Riemann surface with $e(\Sigma,\infty)$ points removed.
 \end{proof}

 \begin{rmk}The surfaces in Theorem \ref{finite topology} have finite topology and finite total curvature, but it is impossible to dominate their topology or total curvature by the Willmore energy and density of such surfaces. For example,  Hoffman and Meeks find\cite{HM90} there are a family of embedded minimal surfaces with three multiplicity one ends but  arbitrary many genuses. Their total curvature also tend to infinity as the genus goes to infinity.
 \begin{figure}[!htbp]
	\centering
	\begin{tabular}{c}
		\includegraphics[width=0.30\linewidth]{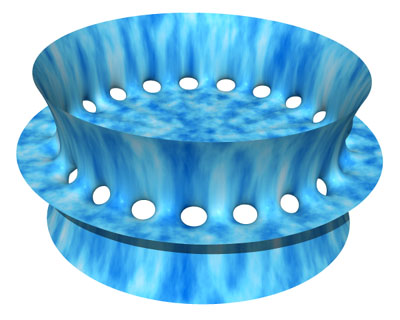}\\
		The Costa-Hoffman-Meeks surface with many handles
	\end{tabular}
\end{figure}
 \end{rmk}

\section{Applications}\label{realapplication}
\subsection{Isolated Singularities}
\

In this subsection, we will care about the isolated singularity and do inverse process of section \ref{application}.
\begin{prop}\label{isodensityidentity}
Assume $\Sigma\subset B_1(0)\backslash\{0\}\subset\mathbb{R}^{2+k}$ is a properly immersed surface with $\partial{\Sigma}\subset \partial B_1(0)$,
\begin{align}\label{isofinitewill}
\int_{\Sigma\backslash \{0\}}|H|^2d\mathcal{H}^2<+\infty
\end{align}
and
\begin{align*}
\Theta_*(\Sigma,0)=\liminf_{r\to 0}\frac{\mathcal{H}^2(\Sigma\cap B_r(0))}{\pi r^2}<+\infty.
\end{align*}
Then the inverted surface  $\tilde{\Sigma}$ is properly immersed in $\mathbb{R}^{n+k}$ with finite density $\Theta(\tilde{\Sigma},\infty)\ge 1$ at infinity and
\begin{align}\label{isofinitewillmore}
\int_{\tilde{\Sigma}}|\tilde{H}|^2d\mathcal{H}^2<+\infty.
\end{align}
Moreover, there holds the density identity
\begin{align}\label{isofinitedensity}
\Theta(\Sigma,0)=\Theta(\tilde{\Sigma},\infty),
\end{align}
which means both sides are well-defined and they are equal.
\end{prop}
\begin{proof}
Since $\partial \Sigma\subset \partial B_1(0)$ is compact, we can close it up and assume $\Sigma\subset B_2(0)$ is a surface without boundary.  By (\ref{isofinitewillmore}) and (\ref{isofinitedensity}) and the same argument as in Proposition \ref{weak removable}, we know
$\bar{\Sigma}=\Sigma\cup \{0\}$ is an integral varifold in $B_2(0)$ with generalized mean curvature $H\in L^2$. So, the monotonicity formula (\ref{monotonicity equality}) holds and $\Theta(\Sigma,0)\ge 1$ is well defined . Noting $\bar{\Sigma}$ has finite volume, by Corollary \ref{integral discription of upper density at infinity}, we know
\begin{align*}
\int_{\Sigma}\big|\frac{\nabla^{\bot}r}{r}\big|^2d\mathcal{H}^2<+\infty.
\end{align*}
Hence  by letting $\sigma\to 0$ and $\rho\to \infty$ in (\ref{monotonicity equality}), we get
\begin{align}\label{isodensityformula}
\pi\Theta(\Sigma,0)=\int_{\Sigma}\bigg(\frac{|H|^2}{16}-\big|\frac{H}{4}+\frac{\nabla^{\bot}r}{r}\big|^2\bigg)d\mathcal{H}^2.
\end{align}
Also use $f:\Sigma\to B_2(0)\subset \mathbb{R}^{2+k}$ to denote the immersion map and let $h=\frac{f}{|f|^2}$ be the inversion. Again by the observation (\ref{observation}). We know for any $R>0$,
\begin{align*}
\int_{\tilde{\Sigma}\cap B_R(0)}|\tilde{H}|^2d\mathcal{H}^2\le 320\int_{\Sigma\backslash B_{\frac{1}{R}(0)}}\bigg(|H|^2+\big|\frac{\nabla^{\bot}r}{r}\big|^2\bigg)d\mathcal{H}^2.
\end{align*}
Letting $R\to \infty$, we get $\int_{\tilde{\Sigma}}|\tilde{H}|^2d\mathcal{H}^2<+\infty$ and the monotonicity formula (\ref{monoton}) holds for $\tilde{\Sigma}$.

Now, on the one hand, since $\tilde{\Sigma}\subset \mathbb{R}^{2+k}\backslash B_{\frac{1}{2}}(0)$, we know for $\sigma<\frac{1}{2}$,
\begin{align*}
\frac{\mathcal{H}^2(\tilde{\Sigma}\cap B_{\sigma}(0))}{\sigma^2}=\frac{1}{2\sigma^2}\int_{B_{\sigma}(0)}\langle \tilde{r}\tilde{\nabla}^{\bot}\tilde{r}, \tilde{H}\rangle d\mathcal{H}^2=0.
\end{align*}
On the other hand, by (\ref{bracketterm}), we know
\begin{align*}
|\langle \tilde{\nabla}^{\bot}\tilde{r}, \tilde{H}\rangle|
\le 4|f^{\bot}|^2+|f|^2|\langle H,f^{\bot}\rangle|
\le 5|f^{\bot}|^2+ |H|^2|f|^4.
\end{align*}
So,
\begin{align*}
\big|\frac{1}{2\rho^2}\int_{\tilde{\Sigma}\cap B_{\rho}(0)}\langle \tilde{r}\tilde{\nabla}^{\bot}\tilde{r}, \tilde{H}\rangle d\mu_{\tilde{g}}=\big|
&\le \frac{5}{2\rho^2}\int_{\Sigma\backslash B_{\frac{1}{\rho}}(0)}\frac{1}{|f|}(|f^{\bot}|^2+|f|^4|H|^2)\frac{1}{|f|^4}d\mu_g\\
&\le \frac{5}{2\rho} \int_{\Sigma}\bigg(|H|^2+\big|\frac{\nabla^{\bot}r}{r}\big|^2\bigg)d\mu_g.
\end{align*}
Letting $\sigma\to 0$  and $\rho\to \infty$ in (\ref{monoton}), we get
\begin{align*}
\Theta(\tilde{\Sigma},\infty)
&=-\lim_{\rho\to\infty,\sigma\to 0}\frac{1}{\pi}\int_{\tilde{\Sigma}\cap (B_\rho\backslash B_\sigma)}
\bigg(\frac{|\tilde{H}|^2}{16}-\big|\frac{\tilde{H}}{4}
+\frac{\tilde{\nabla}^{\bot}\tilde{r}}{\tilde{r}}\big|^2 \bigg)d\mu_{\tilde{g}}\\
&=\lim_{\rho\to\infty,\sigma\to 0}\frac{1}{\pi}\int_{\Sigma\cap (B_{\frac{1}{\sigma}}\backslash B_{\frac{1}{\rho}})}
\bigg(\frac{|H|^2}{16}-\big|\frac{H}{4}
+\frac{\nabla^{\bot}r}{r}\big|^2 \bigg)d\mu_{g}\\
&=\Theta(\Sigma,0),
\end{align*}
where we use the local antisymmetric transformation formula (\ref{localinvariant}) and (\ref{isodensityformula}).
\end{proof}
Similar to the conception of the number of ends at infinity,  for a surface $\Sigma$ properly immersed in $B_1(0)\backslash\{0\}$, we define the number of local connected components of $\Sigma$ near $0$ by
\begin{align*}
e(\Sigma, 0)=\lim_{r\to 0}\tilde{\beta}_0(\Sigma\cap B_r(0)\backslash\{0\})
\end{align*}
where by $\tilde{\beta}_0$ we mean the number of noncompact connected components of a topology space.
\begin{cor}\label{removability}
Assume $\Sigma\subset B_1(0)\backslash\{0\}\subset\mathbb{R}^{2+k}$ is a properly immersed surface with $\partial{\Sigma}\subset \partial B_1(0)$ and satisfying (\ref{isofinitewill}) and
\begin{align*}
e(\Sigma,0)>\Theta_*(\Sigma,0)-1<\infty.
\end{align*}
Then $\Sigma$ has finite topology and finite total curvature. Moreover, we also know $\Theta(\Sigma,0)=e(\Sigma,0)$ is an integer and for small $r>0$, and $1\le i\le e(\Sigma,0)$, each component $\big((\Sigma_i\cup \{0\})\cap B_r(0),g)$  is bi-Lipschitz  homeomorphic to a $2$-dimensional disk.
\end{cor}
\begin{proof}Without loss of generality, we assume $e(\Sigma,0)=1$ and $\Theta_*(\Sigma,0)<2$. By Proposition \ref{isodensityidentity}, the inverted surface $\tilde{\Sigma}$ is properly immersed in $\mathbb{R}^{2+k}$ with finite Willmore energy and
$$1\le \Theta(\tilde{\Sigma},\infty)=\Theta(\Sigma,0)<2.$$
So, by Theorem \ref{finite topology}, $\tilde{\Sigma}$ has finite topology and finite total curvature, is conformal to a punctured disk when restricted to the outside of a large ball and has density $\Theta(\tilde{\Sigma},\infty)=1.$
So, $\Sigma\cap B_r(0)$ is conformal to a punctured disk for small $r$, i.e., there is a conformal parametrization $\varphi: D_1(0)\backslash\{0\}\to \Sigma\cap B_r(0)$.   Noting the trace free part of the second fundamental form is conformal invariant,  We know
\begin{align*}
\int_{\Sigma}|A|^2d\mu_g
&\le 2\int_{\Sigma}|A-\frac{H}{2}g|^2d\mu_g+\int_{\Sigma}|H|^2d\mu_g\\
&=2\int_{\tilde{\Sigma}}|\tilde{A}-\frac{\tilde{H}}{2}\tilde{g}|^2+\int_{\Sigma}|H|^2d\mu_g\\
&\le 4\int_{\tilde{\Sigma}}|\tilde{A}|^2d\mu_{\tilde{g}}
+2\int_{\tilde{\Sigma}}|\tilde{H}|^2d\mu_{\tilde{g}}+\int_{\Sigma}|H|^2d\mu_g<+\infty.
\end{align*}
By Kuwert and Li's  classification theorem\cite[Theorem 3.1]{KL12} for isolated singularities  of surfaces with finite area and finite total curvature, we know $\varphi\in W^{2,2}(D,\mathbb{R}^{2+k})$ and the induced conformal metric $g=e^{2u}(dx^2+dy^2)$ satisfying
$$u(z)=m\log{z}+w(z)$$
for $w(z)\in C^0\cap W^{1,2}(D)$ and $m=\Theta(\Sigma,0)-1$. Now, since $\Theta(\Sigma,0)=\Theta(\tilde{\Sigma},\infty)=1$, we know $m=0$.  Hence $u=w\in C^0(D)$, which means
\begin{align*}
\frac{1}{C}|x-y|\le d_g(f(x),f(y))=\inf_{\gamma\text{ joining } x, y} \int_{0}^{1}e^{u(\gamma)}|\dot{\gamma}|dt\le C|x-y|,
\end{align*}
i.e., $f: D\to (\bar{\Sigma}\cap B_r(0),g)$ is a bi-Lipschitz parametrization for $r$ small.
\end{proof}
\begin{rmk}
The same conclusion holds for surfaces properly immersed in a punctured geodesic ball $ B_1(p)\backslash\{p\}$ of a Riemannian manifold $(M^{2+k},g)$, since $(M,g)$ can be embedded in $R^{2+k+N}$ by Nash embedding theorem and the density, topology and finiteness of Willmore energy of $\Sigma$ will not change.
\end{rmk}
\subsection{Uniqueness of The Catenoid and Minimal Ends}\label{minimalend}
\

As a corollary, we prove a uniqueness result for the catenoid.
 \begin{cor}\label{catenoid}
   Assume $\Sigma\subset \mathbb{R}^3$ is a connected properly immersed minimal surface with at least two ends. If
   $$\Theta(\Sigma,\infty)<3,$$
   then $\Sigma$ is the catenoid.
   \end{cor}
\begin{proof}
Since $e(\Sigma,\infty)\ge 2>\Theta(\Sigma, \infty)-1$, by Theorem \ref{finite topology}, we know $\Sigma$ has finite total curvature and exactly two embedded ends. So, by Schoen's uniqueness theorem\cite{S83}, $\Sigma$ is a catenoid.
\end{proof}

As mentioned in the introduction, this uniqueness of the catenoid is also a direct corollary of Leon Simon's theorem on the uniqueness of the tangent cone\cite{LS83b}\cite[The paragraph after Theorem 5.7]{LS85}. The following is a most simple example of such uniqueness phenomenon.

 \begin{cor}\label{global} Assume $\Sigma$ is a complete immersed minimal surface in $\mathbb{R}^{2+k}$ with $$\Theta(\Sigma,+\infty)<e+1\ \ \ \text{ and }\ \ \ \  e(\Sigma,\infty)\ge e.$$ Then $\Sigma$ has exactly $e$ ends and each end $\Sigma_i$ can be written as a graph over some plane $V_i$ in with gradient tends to be zero. Moreover, in the case $k=1$, these $T_i$ are the same.
\end{cor}

\begin{proof}
Since $\Sigma$ is complete and of quadratic area growth, by \cite[Lemma 3]{Ch97}, the immersion $f$ is proper. By Theorem \ref{finite topology}, there exist $r_1>0$ such that
$$\Sigma\backslash  B_{r_1}(0)=\sqcup_{i=1}^{e}\Sigma_i,$$
where $e=e(\Sigma,\infty)$ and each $\Sigma_i$ is conformal to a punctured disk with finite total curvature and $\Theta(\Sigma_i,\infty)=1$.  Moreover, since $\Sigma_i$ is  minimal, its Gaussian map $G(x)=e_1(x)\wedge e_2(x):\Sigma_i \to (G_{2,n}(\mathbb{R}),g_c)$ is a harmonic map on the punctured disk with finite energy(note the energy of the Gaussian map is exactly the total curvature). So by Sacks and Uhlenbeck's \cite[Theorem 3.6]{SU81} removability of singularity for harmonic maps with finite energy(or \cite[Theorem A]{H91a}), $G(x)$ can be extended continuously across infinity. The rest is well known. 

\end{proof}

\end{document}